\documentclass[10pt]{amsart}
\usepackage{amssymb}
\usepackage{amsmath}
\usepackage{amsfonts}
\usepackage{mathrsfs}
\usepackage{psfrag}
\usepackage{verbatim}
\usepackage{color}
\usepackage[usenames,dvipsnames]{xcolor}
\usepackage{graphicx} 
\usepackage{bm}
\usepackage{bbm}
\usepackage{dsfont}
\usepackage{enumerate}
\usepackage{marginnote}
\usepackage{comment}
\usepackage{subcaption}

\newcommand{\ind}{\mathds{1}}
\usepackage[foot]{amsaddr}
\usepackage[left]{lineno}

\usepackage[colorlinks=true]{hyperref}
\hypersetup{citecolor=MidnightBlue}
\DeclareGraphicsExtensions{}
\theoremstyle{plain}
\newtheorem{theorem}{Theorem}[section]

\newtheorem{lemma}[theorem]{Lemma}
\newtheorem{proposition}[theorem]{Proposition}

\theoremstyle{definition}
\newtheorem{assumption}[theorem]{Assumption}
\newtheorem{definition}[theorem]{Definition}
\newtheorem{remark}[theorem]{Remark}
\theoremstyle{remark}
\newtheorem{notation}[theorem]{Notation}

\newcommand{\eps}{\varepsilon}

\newcommand{\la}{\left\langle}
\newcommand{\ra}{\right\rangle}

\newcommand{\BR}{\mathbb{R}}
\newcommand{\BP}{\mathbb{P}}
\newcommand{\BE}{\mathbb{E}}
\newcommand{\BZ}{\mathbb{Z}}
\newcommand{\BN}{\mathbb{N}}

\newcommand{\sfT}{\mathsf{T}} 
\newcommand{\itd}{\mathtt{Id}} 
\newcommand{\sfN}{\mathsf{N}} 
\newcommand{\sfM}{\mathsf{M}}
\newcommand{\sfu}{\mathsf{u}} 
\newcommand{\Sk}{\mathsf{Sk}}

\newcommand{\sfK}{\mathsf{K}} 
\newcommand{\sfF}{\mathsf{F}}

\newcommand{\II}{\mathtt{I}}
\newcommand{\JJ}{\mathtt{J}}

\newcommand{\LL}{\mathtt{L}}
\newcommand{\MM}{\mathtt{M}}
\newcommand{\NN}{\mathtt{N}}

\newcommand{\ren}{\mathtt{q}}

\newcommand{\filt}{\mathscr{F}}

\newcommand{\sfz}{\mathsf{z}}

\title{Random perturbations of systems with periodic impulse effects}
\author[Ashif Khan and Chetan D. Pahlajani]{Ashif Khan$^{1,*}$ and Chetan D. Pahlajani$^{1}$}

\address{$^{1}$Department of Mathematics\\ Indian Institute of Technology Gandhinagar, Palaj, Gandhinagar 382355, India}
\thanks{${}^{*}$Corresponding author}

\thanks{The first author would like to thank Government of India for financial support through the Prime Minister's Research Fellowship (PMRF) scheme. The second author's research was supported by ANRF project number MTR/2023/000545.}

\email{khanashif@iitgn.ac.in, cdpahlajani@iitgn.ac.in}
\date{\today.}

\begin{document}
\maketitle
     {\small \centerline{$^{1}$Department of Mathematics, Indian Institute of Technology Gandhinagar, Palaj, Gandhinagar 382355, India}}
     
     	\begin{abstract}
		The principal aim of the present work is to explore limit theorems for small random perturbations of periodically kicked dynamical systems, in the limit of vanishing noise intensity. 
		Recasting our system, which is described heuristically by an ordinary differential equation forced by a periodic train of delta spikes, in terms of an impulsive system with resetting, we consider small state-dependent Brownian perturbations of this system and explore the zero noise limit on finite, but arbitrary, time horizons. 
		For the resulting nonlinear stochastic system with impulse effects, we prove convergence to the underlying deterministic nonlinear impulsive system as the noise goes to zero.
		More importantly, we prove convergence of the rescaled fluctuation process about the deterministic limit, in a strong pathwise sense on finite time intervals, to a limiting fluctuation process governed by a linear time-dependent stochastic differential equation in between impulses and a linear time-dependent resetting map at impulses.
In each of these cases, we are able to obtain rigorous estimates on the remainders.
Our results thus provide quantitative tools to accurately and efficiently approximate small random fluctuations of the original nonlinear stochastic impulsive system about the deterministic limit in terms of linear, albeit time-inhomogeneous, surrogates.		
The results are illustrated numerically for two different examples: a periodically kicked nonlinear pendulum with state-dependent kick sizes and a FitzHugh-Nagumo system subject to periodic kicks.
	\end{abstract}

\section{Introduction}\label{S:Introduction}

Many dynamical systems of interest take the form of ordinary differential equations ({\sc ode}) subject to external time-dependent perturbations, both due to deterministic influences acting on the system and random perturbations due to fluctuating external forces. 
Of these, the deterministic influences often take the form of time-periodic forcing, especially when studying oscillators \cite{GuckenheimerHolmes}.
In many scenarios, the time-periodic forcing is sharply localized and can be modelled by \textit{periodic kicks} \cite{WY_CMP2003}, \cite{LinYoung_JFPTA}, \cite{Balasuriya_Nonlinearity}, leading to impulsive {\sc ode} driven by a train of delta functions of the form 
\begin{equation}\label{E:kicked-ode}
\dot{x}=b(x) + g(x) \sum_{k=1}^\infty \delta(t-t_k), \quad x(0) = x_0 \in \BR^d.
\end{equation}
Here, $x(t) \in \BR^d$ denotes the state of the system at time $t \in [0,\infty)$, $b,g:\BR^d \to \BR^d$ are sufficiently regular vector fields, with $b$ describing the nominal dynamics and $g$ describing the strength of the impulses $\delta(t-t_k)$ at the known impulse times $\{t_k\}_{k=1}^\infty$.  
Following \cite{WY_CMP2003}, \cite{LinYoung_JFPTA}, solutions $x(t)$ to \eqref{E:kicked-ode} are interpreted in terms of {\sc ode} subject to \textit{resetting} at the impulse times $\{t_k\}_{k=1}^\infty$, viz., as solutions of 
\begin{equation}\label{E:sie-det-per}
\begin{aligned}
\dot{x}(t) &= b(x(t)) \quad \quad \text{for $t \in [0,\infty)\setminus\{t_k\}_{k=1}^\infty$}\\x^+(t_k) &= h(x^-(t_k)) \qquad \text{for $k \ge 1$,} \qquad x(0)=x_0 \in \BR^d.
\end{aligned}
\end{equation}
Here, $h:\BR^d \to \BR^d$ is a suitable \textit{resetting map} which captures the impulsive effects of $g$ (see Remark \ref{R:pks}), and $x^+(t)$, $x^-(t)$ correspond to the right- and left-hand limits of $x(t)$ defined by 
$x^+(t)\triangleq \lim_{s \searrow t}x(s)$ for all $t \ge 0$, $x^-(t)\triangleq \lim_{s \nearrow t}x(s)$ for all $t>0$, $x^-(0) \triangleq x(0)$.
We assume that the resulting piecewise-continuous trajectories are right-continuous at the impulse times $t_k$, i.e., $x^+(t_k)=x(t_k)$ for all $k$.\\ 

Periodically kicked systems of the type described above are of significant importance, both from the theoretical and application viewpoints.
Indeed, equations of the type \eqref{E:kicked-ode}, or equivalently \eqref{E:sie-det-per}, are used to study the creation and nature of chaos in various periodically forced dynamical systems \cite{WY_CMP2003}, \cite{LinYoung_JFPTA}, \cite{RyalsYoung}, \cite{Ott2008}, \cite{MohapatraOtt}. 
Applications also arise in neuroscience \cite{Lin_SIADS2006}, \cite{LajoieShea-Brown}, \cite{AmbrosioMintchev} where the state vector $x(t) \in \BR^d$ might keep track of the membrane voltage together with gating or recovery variables, with the vector field $b(x)$ encoding neural (e.g., Hodgkin-Huxley or FitzHugh-Nagumo) dynamics, and the impulsive forcing $g(x) \sum_{k=1}^\infty \delta(t-t_k)$ serving as a model for time-localized inputs (with state-dependent magnitudes) from other neurons. 
Similar, or related, models occur in mechanical systems with impacts \cite{Brogliato-book} such as walking robots \cite{GAP} or percussive drilling \cite{DDD}, and also in impulsive control problems where a control action or forcing function is applied in extremely short intense bursts \cite{KantMukherjee_SysConLett2020}, \cite{KantMukherjeeKhalil_NonlinDyn2021}, \cite{JCS_OptimalSpacecraftGuidance}.
Finally, kicked oscillators with nonlinear resetting also arise in numerous physics problems, both in the classical and quantum settings; see \cite{MuddeJansz} for an example of the former, together with a pointer to relevant references dealing with the latter.
\\

In most realistic settings, and certainly in the case of the applications to mechanics and neuroscience alluded to above, systems of interest are subject to random perturbations, stemming from either random external fluctuations, or from unmodeled dynamics, or both.
The natural approach to capture these effects is to replace the {\sc ode} above by a stochastic differential equation ({\sc sde}) driven by Brownian motion \cite{KS91}, \cite{Oksendal}. 
If the impulsive system \eqref{E:sie-det-per} is now perturbed by state-dependent white noise of size $\eps$ with $0<\eps \ll 1$, then the state vector $X^\eps_t$ becomes a stochastic process solving the nonlinear {\sc sde} with time-periodic resetting given by
\begin{equation}\label{E:sie-stoch-per}
\begin{aligned}
dX^\eps_t &= b(X^\eps_t)\thinspace dt + \eps \sigma(X^\eps_t) \thinspace dW_t \qquad \text{for $t \in [0,\infty)\setminus\{t_k\}_{k=1}^\infty$,}\\X^{\eps,+}_{t_k} &= h(X^{\eps,-}_{t_k}) \qquad \qquad \qquad \medspace \qquad \text{for $k \ge 1$,} \qquad X^\eps_0=x_0 \in \BR^d.
\end{aligned}
\end{equation}
Here, $W_t$ is an $r$-dimensional Brownian motion on some probability space $(\Omega,\filt,\BP)$, and the function $\sigma:\BR^d \to \BR^{d\times r}$, quantifying state-dependence of the random perturbations, is assumed to be sufficiently regular.  
As before,
$X^{\eps,+}_t \triangleq \lim_{s \searrow t}X^\eps_s$ for $t \ge 0$, $X^{\eps,-}_t \triangleq \lim_{s \nearrow t}X^\eps_s$ for $t>0$, $X^{\eps,-}_0 \triangleq X^\eps_0$
are the right- and left-continuous modifications of $X^\eps_t$, defined path-by-path in the manner described above. Once again, the paths of $X^\eps_t$ are chosen to be right-continuous for all $t \ge 0$, with left limits for $t>0$. \\	

The principal question of interest now becomes getting a \textit{quantitative} understanding of the stochastic process $X^\eps_t$ solving \eqref{E:sie-stoch-per}, with a particular focus on the \textit{fluctuations} of $X^\eps_t$ about $x(t)$.
Noting at the outset that explicit analytical solutions to \eqref{E:sie-det-per} and \eqref{E:sie-stoch-per} are essentially out of question, one may next attempt to exploit smallness of the parameter $\eps$ to gain insight into system behavior. 
Expecting, and often confirming, that for any fixed $\sfT>0$, we have $\BE\left[\sup_{0 \le t \le \sfT}|X^\eps_t-x(t)|\right] = \mathscr{O}(\eps)$ as $\eps \searrow 0$, leads to the more refined question: 
What do typical fluctuations of $X^\eps_t$ about $x(t)$ of size $\eps$ look like, or alternatively, how does one efficiently and accurately capture the dynamics of the rescaled fluctuation process $Z^\eps_t \triangleq \frac{X^\eps_t - x(t)}{\eps}$.
While one can certainly choose specific values for the problem parameters $b$, $h$, $\{t_k\}_{k=1}^\infty$, $x_0$, and use numerical methods to probe these questions, it is natural to ask whether one can develop \textit{analytic} methods which exploit smallness of the noise to develop simple, accurate approximations for the fluctuation process $Z^\eps_t$, which are uniformly valid across a \textit{wide range} of parameter values, including choices of vector fields.
This, in a nutshell, is the primary focus of this work.\\

To place this problem in context, we note that in the smooth (impulse-free) case, i.e., in the case where $g(x) \equiv 0$ in \eqref{E:kicked-ode} or $h(x) \equiv x$ in \eqref{E:sie-det-per} above, classical results of \cite{FW_RPDS, Blagoveshchenskii} identify the cofficients in small noise expansions of the form $X^\eps_t = X^0(t) + \eps X^1_t + \eps^2 X^2_t + \dots + \eps^k X^k_t + R^\eps_{k+1}(t)$, while also providing remainder estimates asserting that for any fixed $\sfT>0$, $\BE\left[\sup_{0 \le t \le \sfT}|R^\eps_{k+1}(t)|\right]=\mathscr{O}(\eps^{k+1})$ as $\eps \searrow 0$.\footnote{Of course, these results are predicated on the drift and diffusion coefficients being sufficiently smooth.} 
The zeroth order term $X^0(t)$ is found to coincide with the {\sc ode} in the noiseless case, while the first order term $X^1_t$ is found to satisfy the {\sc sde} $dZ_t = Db(x(t))Z_t \thinspace dt + \sigma(x(t)) \thinspace dW_t$, $Z_0=0$ obtained by linearizing around $x(t)$.
Evidently, the {\sc sde} for $Z_t$ is linear in $Z_t$ and independent of $\eps$, albeit time-inhomogeneous (on account of dependence on $x(t)$).\\

In the present work, we study the nonlinear impulsive systems \eqref{E:sie-det-per}, \eqref{E:sie-stoch-per} subject to some basic assumptions (see Assumptions \ref{A:vf}, \ref{A:resetting-map}) and address the following questions: (i) How can one carry out a program similar to that outlined above for the impulsive case where $h(x) \not\equiv x$, and most importantly, (ii) How should the effect of the possibly nonlinear resetting map $h$ be codified in the limiting fluctuation process $Z_t$.
To probe these questions, we start by showing that, as in the classical case, for any fixed $\sfT>0$, we have $\BE\left[\sup_{0 \le t \le \sfT}|X^\eps_t-x(t)|\right] = \mathscr{O}(\eps)$ as $\eps \searrow 0$, a result that can be interpreted as a Law of Large Numbers.
More interestingly, we are able to show that if $Z_t$ is the ($\eps$-independent) process solving the linear {\sc sde} $dZ_t = Db(x(t))Z_t \ dt+ \sigma(x(t)) \thinspace dW_t$ \textit{between} impulses, and is governed by the resetting condition $Z^{+}_{t_k}=Dh(x^-(t_k)) Z^{-}_{t_k}$ \textit{at} impulses, then for any fixed $\sfT>0$, we have $\BE\left[\sup_{0 \le t \le \sfT}|X^\eps_t-x(t)-\eps Z_t|\right]=\mathscr{O}(\eps^2)$ as $\eps \searrow 0$.
The immediate corollary $\BE\left[\sup_{0 \le t \le \sfT}\left|\frac{X^\eps_t-x(t)}{\eps}-Z_t\right|\right]=\mathscr{O}(\eps)$ as $\eps \searrow 0$ admits interpretation as a sort of Central Limit Theorem for the rescaled fluctuation process $Z^\eps_t=\frac{X^\eps_t-x(t)}{\eps}$.
Once again, the dynamics of $Z_t$ are linear, but time-inhomogeneous, on account of the dependence of both the {\sc sde} and the resetting law on $x(t)$.
The results thus confirm that, much like the {\sc sde} for $Z_t$, the resetting law for $Z_t$ also involves a linearization about the deterministic trajectory $x(t)$. 
Our principal contribution is thus to consider a general (potentially nonlinear) resetting map $h$ and carefully quantify, to leading order, the combined effect of alternating flow ({\sc sde}) and resetting (via the map $h$) on the evolution of the error process $X^\eps_t - x(t)$, thereby extending the results of \cite{Blagoveshchenskii}.
It is worth noting that, in addition to providing simpler linear approximations to the original nonlinear system, our results---or at the very least, similar linearization techniques---may be adaptable to treat problems of optimal control or uncertainty quantification in nonlinear stochastic settings, as exemplified by \cite{JCS_OptimalSpacecraftGuidance}, \cite{BlakeMacleanBalasuriya}. \\

To better situate our work within the context of the existing literature, we start by noting that periodically kicked systems, or systems with impulse effects ({\sc sie}) more generally, form a class of nonsmooth dynamical systems \cite{HaddadChellaboinaNersesov+2014}, \cite{Filippov}, \cite{BristolBook} with \textit{hybrid} behavior; by the latter we mean a system whose evolution involves the close interplay between continuous and discrete dynamics \cite{GST-HybDynSys-book}.
Mathematically speaking, within {\sc sie}, one can distinguish between problems with time-dependent versus state-dependent impulse arrivals.
For the first of these categories, the impulse times are known \textit{a priori}, or at any rate, are independent of the state.
This includes, in addition to the present work,  \cite{WY_CMP2003}, \cite{LinYoung_JFPTA}, \cite{DDD}, \cite{JCS_OptimalSpacecraftGuidance} where the impulses, whether originating due to kicks or control action, arrive in a time-periodic manner.
The second category includes cases where impulses, whether due to impact or control action, occur when the state hits a suitable switching manifold, e.g., \cite{GAP}, \cite{VeerRakeshPoulakakis}, \cite{KantMukherjee_SysConLett2020}, \cite{KantMukherjeeKhalil_NonlinDyn2021}, \cite{Moehlis_PrecisionNoisyOscillators}, \cite{LHJ_Walking}.\\

One notes that obtaining small noise expansions, as we have done in this work, constitutes a tool to analytically describe (in reasonably closed form) \textit{typical} fluctuations of $X^\eps_t$ about $x(t)$, viz., fluctuations of size $\eps$.
Other asymptotic tools of interest include large deviations \cite{FW_RPDS,DZ98} and averaging \cite{FW_RPDS,SHS-RPM}; the former focuses on estimating rate of decay (in $\eps$) of probabilities of \textit{rare} (extremely atypical) events, while the latter deals with questions of model reduction in the presence of multiple time scales. 
While all these techniques have been extensively explored for randomly perturbed smooth dynamical systems, the corresponding theory for non-smooth (including impulsive) systems, is much less developed and has recently been the subject of much interest; see, for instance, \cite{SNP}, \cite{BOQ}, \cite{JeffreySimpson_NonFilippovSmoothing_NonlinDyn_2014}, \cite{HZG_Physica-D2022}. 
For systems with impacts/impulse effects subject to stochastic perturbations, questions of averaging, bifurcations, steady-state behavior have also been explored; see, for instance, \cite{DI}, \cite{RounakGupta}, \cite{SNP}, \cite{Zhu_PoissonImpulses}, and the references therein. 
\\

The rest of the paper is organized as follows. In Section \ref{S:ProblemStatement}, we precisely formulate our problem of interest and state our main results, viz., Theorems \ref{T:LLN} and \ref{T:CLT}. 
Section \ref{S:LLN} is devoted to the proof of Theorem \ref{T:LLN}, while Section \ref{S:CLT} tackles the proof of Theorem \ref{T:CLT}. Finally, in Section \ref{S:Example}, we numerically illustrate our results on two different examples: a periodically kicked nonlinear pendulum with nonlinear resetting map, and a Fitzhugh-Nagumo system subject to periodic kicks.
\\

\subsection*{Some comments regarding notation}
We close out this section by listing some frequently used notation and recording a few notational conventions. 
For \( x \in \mathbb{R}^n \), $n \ge 1$, let \( \|x\| \triangleq \left( \sum_{i=1}^{n} |x_i|^2 \right)^{1/2} \) denote the standard Euclidean norm on $\BR^n$, and for $A=[a_{ij}] \in \BR^{m\times n}$, let $\|A\|$ and $\|A\|_\sfF$ be, respectively, the corresponding induced matrix norm and the Frobenius norm, i.e., 
$\|A\| \triangleq \max_{\substack{x \in \BR^n\\x \neq 0}} \frac{\|Ax\|}{\|x\|}$ and $\|A\|_\sfF\triangleq\sqrt{\sum_{i=1}^{m}\sum_{j=1}^n|a_{ij}|^2}=\mathsf{Trace}(A^{\intercal}A)$.
Since 
$\|A\| \le \|A\|_\sfF$, we have 
$\|Ax\| \le \|A\| \|x\| \le \|A\|_\sfF \|x\|$ for every $x \in \BR^n$.
Note that both the matrix norms $\|\cdot\|$ and $\|\cdot\|_\sfF$ are \textit{submultiplicative}, i.e., $\|AB\| \le \|A\| \|B\|$, $\|AB\|_\sfF \le \|A\|_\sfF \|B\|_\sfF$ for all matrices $A$, $B$ such that the products are well-defined. 
The indicator function of the set $A$ will be denoted by $\ind_A$.
For real numbers $a,b$, the maximum and minimum of $a$ and $b$ will be denoted by $a \vee b$ and $a \wedge b$, respectively.
As a matter of convention, we will denote generic constants in various estimates by $C$, allowing the exact value of $C$ to  change from line to line. 
Further, with the sole exception of $\eps$, we will absorb the dependence on \textit{all} problem parameters into $C$. 
Thus, a typical $C$ will depend on all problem parameters (including time horizon $\sfT$), but \textit{not} on $\eps$. 
Finally, we mention a comment regarding multiplication of matrices using product notation. 
Let $A_1,A_2,\dots$ be $d \times d$ matrices. For $1 \le I \le J$, we will always understand $\prod_{k=I}^J A_k$ to be obtained by taking successive terms (as $k$ increases) on the left while multiplying matrices; thus, 
\begin{equation}\label{E:matrix-mult-conv}
\prod_{k=I}^J A_k \triangleq A_J \medspace A_{J-1} \dots A_{I+1} \medspace A_I; \qquad \text{and we let} \qquad \sum_{i=1}^0 (\dots) \triangleq \bm 0, \quad \text{and} \quad \prod_{j=1}^0 (\dots) \triangleq I_d
\end{equation}
where $\bm 0$ and $I_d$ denote the $d \times d$ zero and identity matrices respectively.

\section{Problem Formulation and Statement of Main Results}\label{S:ProblemStatement}

We start by making precise the assumptions on various terms appearing in our impulsive systems \eqref{E:sie-det-per}, \eqref{E:sie-stoch-per} of interest, viz., the drift $b$, the dispersion matrix $\sigma$, the resetting map $h$,\footnote{In Remark \ref{R:pks} below, we discuss how to transition from $g$ in \eqref{E:kicked-ode} to $h$ in \eqref{E:sie-det-per} following the arguments in \cite{WY_CMP2003}.} and the impulse times $\{t_k\}_{k=1}^\infty$.
In regards to the latter, we will assume, without loss of generality, that the first impulse arrives at some fixed time $\alpha \in (0,1]$ with successive impulses spaced $1$ time unit apart.

\begin{assumption}[Regularity of drift and diffusion coefficients]\label{A:vf}
The function $b:\BR^d \to \BR^d$ is $C^2$ with all components having globally bounded first- and second-order partial derivatives. For convenience, we set
\begin{equation}\label{E:K_b}
\sfK_b \triangleq \left(\sup_{x \in \BR^d} \max_{1 \le i,j \le d} \left|\frac{\partial b_i}{\partial x_j}(x)\right|\right) \vee \left(\sup_{x \in \BR^d} \max_{1 \le i,j,k \le d} \left|\frac{\partial^2 b_i}{\partial x_j \partial x_k}(x)\right|\right) \vee 1.
\end{equation}
The function $\sigma:\BR^d \to \BR^{d \times r}$ is globally Lipschitz with linear growth; thus, there exists $\sfK_\sigma>0$ such that
\begin{equation}\label{E:K_sigma}
\|\sigma(x)-\sigma(y)\|_\sfF \le \sfK_\sigma \|x-y\|, \qquad \|\sigma(x)\|_\sfF \le \sfK_\sigma \left(1+ \|x\|\right) \qquad \text{for all $x,y \in \BR^d$.}
\end{equation}
\end{assumption}

\begin{assumption}[Resetting map]\label{A:resetting-map}
The resetting map $h:\BR^d \to \BR^d$ is $C^2$ with all components having globally bounded first- and second-order partial derivatives. For convenience, we set
\begin{equation}\label{E:K_h}
\sfK_h \triangleq \left(\sup_{x \in \BR^d} \max_{1 \le i,j \le d} \left|\frac{\partial h_i}{\partial x_j}(x)\right|\right) \vee \left(\sup_{x \in \BR^d} \max_{1 \le i,j,k \le d} \left|\frac{\partial^2 h_i}{\partial x_j \partial x_k}(x)\right|\right) \vee 1.
\end{equation}
\end{assumption}

\begin{remark}\label{R:linear-ok}
It is worth observing that while Assumptions \ref{A:vf} and \ref{A:resetting-map} impose boundedness of derivatives of $b, h$ and a global Lipschitz condition for $\sigma$, the functions $b, h, \sigma$ themselves are allowed to grow linearly. This is important, since we would like our framework to accommodate linear {\sc sde} with affine resetting maps, given that resetting maps of the form $h(x)=Ax$ with suitable matrix $A$ are ubiquitous in impact problems. 
Also, we have ensured that $\sfK_b$ and $\sfK_h$ are greater than or equal to $1$ in order to simplify certain calculations.
\end{remark}

\begin{assumption}[Impulse times]\label{A:impulse-times}
Fix $\alpha \in (0,1]$ and define a sequence $\mathscr{T}$ of times by 
\begin{equation}\label{E:impulse-times}
\mathscr{T} \triangleq \{t_k\}_{k=0}^\infty, \quad \text{where} \quad t_0 \triangleq 0, \quad t_k=k-1+\alpha \quad \text{for $k \in \BN$;}
\end{equation}
here, $t_k$ denotes the $k$-th impulse time for each $k \ge 1$.
\end{assumption}
 
Going forward, we will find it helpful to characterize $x(t)$ solving \eqref{E:sie-det-per} and $X^\eps_t$ solving \eqref{E:sie-stoch-per} by the following integral equations with resetting:
\begin{equation}\label{E:det-state-int-eq}
x(t) = \sum_{k \ge 1} \ind_{[t_{k-1},t_k)}(t) \left\{x^+(t_{k-1}) + \int_{t_{k-1}}^t b(x(s)) \thinspace ds \right\},    \text{$x^+(t_{k})=h(x^-(t_{k}))$ for $k \ge 1$, $x^+(t_0)=x_0$,}
\end{equation}
\begin{multline}\label{E:stoch-state-int-eq}
X^\eps_t = \sum_{k \ge 1} \ind_{[t_{k-1},t_k)}(t) \left\{X^{\eps,+}_{t_{k-1}} + \int_{t_{k-1}}^t b(X^\eps_s) \thinspace ds  + \eps\int_{t_{k-1}}^t \sigma(X^\eps_s) \thinspace dW_s\right\},\\   \text{$X^{\eps,+}_{t_k}=h(X^{\eps,-}_{t_k})$ for $k \ge 1$, $X^{\eps,+}_{t_0}=x_0$.} 
\end{multline}
Note that the quantities $x(t)$ and $X^\eps_t$ are right-continuous for all $t \ge 0$, and continuous everywhere \textit{except} $\{t_k\}_{k=1}^\infty$ where they almost inevitably have jumps.\footnote{The only exception is at fixed points of the mapping $h$.} 
In other words, if we fix a time horizon $[0,\sfT]$ with $\sfT \in (0,\infty)$, then the natural space for the trajectory $x(t)$ and the sample paths of both the stochastic process $X^\eps_t$ and the expected limiting fluctuation process $Z_t$ (alluded to in Section \ref{S:Introduction}) is the function space  $D([0,\sfT];\BR^d)$ consisting of all mappings from $[0,\sfT]$ into $\BR^d$ which are right-continuous for all $t \ge 0$ with left limits for $t >0$. 
Although the space $D([0,\sfT];\BR^d)$ is often topologized using the Skorohod metric $d_\Sk$, we will find it convenient to compare quantities using the uniform metric $d_\sfu(x,y) \triangleq \sup_{t \in [0,\sfT]}\|x(t)-y(t)\|$; see Remark \ref{R:whysupnorm}. 
We now state our first result, whose proof will be provided in Section \ref{S:LLN}. 		
	
\begin{theorem}[Law of large numbers]\label{T:LLN} 
Let $x(t)$ and  $X^\varepsilon_t$ solve the systems $\eqref{E:sie-det-per}$ and $\eqref{E:sie-stoch-per}$ respectively, and fix $p \in \{1,2\}$. Then, for any fixed $\sfT>0$, there exists a constant  $C_{\ref{T:LLN}}(\sfT)>0$ 
such that
\begin{equation}{\label{E:LLN-rate}}
\mathbb{E}\left[\sup_{0 \le t \le \sfT} \|X^\eps_{t}-x(t)\|^p\right]\le C_{\ref{T:LLN}}(\sfT) \thinspace \eps^p \quad \quad \text{for $\eps \in (0,1)$.}
		\end{equation}
\end{theorem}

\begin{remark}\label{R:whysupnorm} 
The Skorohod metric on $D([0,\sfT];\BR^d)$, which renders the space complete and separable \cite{ConvProbMeas, EK86}, is obtained by modifying the uniform metric $d_\sfu$ to allow small time distortions in comparison of functions $x(t)$ and $y(t)$ in the following fashion. If we let $\Lambda_\sfT$ be the family of all strictly increasing continuous bijections $\lambda:[0,\sfT] \to [0,\sfT]$ for which
$\gamma_\sfT(\lambda) \triangleq \sup_{0 \le s <t \le \sfT} \left|\log \frac{\lambda(t)-\lambda(s)}{t-s}\right| <\infty$, then the Skorohod metric is defined by setting 
$d_{\Sk}(x,y) \triangleq \inf_{\lambda \in \Lambda_\sfT}\left\{ \gamma_\sfT(\lambda) \vee \sup_{0 \le t \le \sfT} \|x(t) - y(\lambda(t))\| \right\}$ for $x,y \in D([0,\sfT];\BR^d)$.\footnote{The quantity $\gamma_\sfT(\lambda)$ captures the deviation of the time distortion $\lambda(\cdot)$ from the identity. While both the metrics $d_\Sk$ and $d_\sfu$ defined above naturally depend on the time horizon $[0,\sfT]$, we have suppressed the $\sfT$-dependence in our notation in an attempt to reduce clutter.} 
Since the identity function $\itd(t) \equiv t$ belongs to $\Lambda_\sfT$, and we have $\gamma_\sfT(\itd)=0$, it follows that $d_{\Sk}(x,y) \le d_\sfu(x,y)$ for $x,y \in D([0,\sfT];\BR^d)$. 
It is the above feature of allowing small time distortions that allows one to conclude, for instance, that $\lim_{n \to \infty} d_{\Sk}(\ind_{\left[\frac{1}{2}+\frac{1}{n},1\right]},\ind_{\left[\frac{1}{2},1\right]}) = 0$ in $D([0,1];\BR)$.
However, when comparing paths $x(t)$ and $y(t)$ which have jumps at \textit{common} times, as in Theorem \ref{T:LLN} above and also Theorem \ref{T:CLT} below, the time distortions are of limited utility, thereby making the uniform metric $d_\sfu$ an easier tool to use, which additionally provides stronger control, as noted above.
\end{remark}

Informally, Theorem \ref{T:LLN} implies that for any fixed time horizon $\sfT>0$, we have $X^\eps_t=x(t) + \mathscr{O}(\eps)$ uniformly for $t \in [0,\sfT]$ as $\eps \searrow 0$. 
It is natural to ask whether one can identify any of the higher order terms in an expansion of $X^\varepsilon_t$ in powers of the small parameter $\varepsilon$. 
Thus, we would like to find a stochastic process $Z_t$, independent of $\varepsilon$, such that $X^\varepsilon_t=x(t)+\varepsilon Z_t+\mathscr{O}(\varepsilon^2)$, together with rigorous estimates on the error $\BE\left[\sup_{t \in [0,\sfT]}\|X^\eps_t - x(t) - \eps Z_t\| \right]$.
	To this end, let $Z_t$ be the process solving 
\begin{equation}\label{E:fluct-proc}
\begin{aligned}
Z_t &\triangleq \sum_{k \ge 1} \ind_{[t_{k-1},t_k)}(t) \left\{Z^{+}_{t_{k-1}} + \int_{t_{k-1}}^t Db(x(s)) Z_s \thinspace ds + \int_{t_{k-1}}^t \sigma(x(s)) \thinspace dW_s \right\}, \\
Z^{+}_{t_k}&=Dh(x^-(t_k)) Z^{-}_{t_k} \quad \text{for $k \ge 1$}, \qquad Z^{+}_0 \triangleq 0.
\end{aligned}
\end{equation}
Note that the process $Z_t$ solves the {\sc sde} $dZ_t = Db(x(t))Z_t \ dt+ \sigma(x(t)) \thinspace dW_t$ between impulses, and is governed by the resetting condition $Z^{+}_{t_k}=Dh(x^-(t_k)) Z^{-}_{t_k}$ at impulses.
Closely inspecting \eqref{E:fluct-proc}, we see that the process $Z_t$ is right-continuous at the impulse times $\{t_k\}_{k=1}^\infty$, and thus has sample paths in $D([0,\sfT];\BR^d)$.
We also observe that both the {\sc sde} and the resetting condition for $Z_t$ are non-autonomous (on account of dependence on $x(t)$) and involve \textit{linearizations} about the deterministic trajectory $x(t)$ of the (potentially nonlinear) vector field $b$ and resetting condition $x \mapsto h(x)$. 
	
%
	
We now state our main result concerning leading order behavior of the fluctuations of $X^\eps_t$ about $x(t)$.
	
\begin{theorem}[Central limit theorem]\label{T:CLT}
Let $x(t)$ and  $X^\varepsilon_t$ solve the systems $\eqref{E:sie-det-per}$ and $\eqref{E:sie-stoch-per}$ respectively, and let $Z_t$ be the stochastic process given by \eqref{E:fluct-proc}. Then, for any fixed $\sfT>0$, there exists a constant $C_{\ref{T:CLT}}(\sfT)>0$ such that 
\begin{equation}{\label{E:CLT-rate}}
\mathbb{E}\left[\sup_{0 \le t \le \sfT} \|X^\eps_{t}-x(t)-\eps Z_t\|\right]\le C_{\ref{T:CLT}}(\sfT) \thinspace \eps^2 \quad \quad \text{for $\eps \in (0,1)$.} 
\end{equation}
\end{theorem}

The proof of Theorem \ref{T:CLT} uses Theorem \ref{T:LLN}, and is provided in Section \ref{S:CLT}. 		
We next fix some notation which will help us keep track of the impulses. 

\begin{notation}
Let $\ren(t)$ denote the total number of impulses that have occurred up to (and including) time $t \in [0,\infty)$. It is easily checked that if we set
\begin{equation}\label{E:ren}
\ren(t) \triangleq \max\{n \in \BZ^+:t_n \le t\} \quad \text{for $t \ge 0$, then} \quad	\ren(t) = \lfloor t-\alpha \rfloor+1
\end{equation}
where $\lfloor \cdot \rfloor$ denotes the integer floor function. 
\end{notation}	

We will find it convenient to prove Theorems \ref{T:LLN} and \ref{T:CLT} for the case when the terminal time $\sfT$ is \textit{not} an impulse time, i.e.,   
\begin{equation}\label{E:TN}
\sfT \in (0,\infty)\setminus\{t_k\}_{k=1}^\infty, \qquad \text{and we let} \qquad \sfN \triangleq \ren(\sfT)
\end{equation}
denote the total number of impulses up to time $\sfT$. 
This assumption on $\sfT$ does not entail any loss of generality and the results can easily be extended to the case that $\sfT \in \{t_k\}_{k=1}^\infty$.
Indeed, suppose one has been able to prove Theorems \ref{T:LLN} and \ref{T:CLT} for the case that $\sfT \in (0,\infty)\setminus\{t_k\}_{k=1}^\infty$ and would now like to extend the results to the case $\sfT=t_K$ for some $K \in \BN$.
In this case, one can apply the conclusions of Theorems \ref{T:LLN} and \ref{T:CLT} on a slightly larger time horizon $[0,\tilde\sfT]$ with $\sfT < \tilde\sfT \in (0,\infty)\setminus\{t_k\}_{k=1}^\infty$ and obtain the desired results by noting the nondecreasing (as a function of time $t$) nature of the quantities $\sup_{s \in [0,t]}\|X^\eps_s-x(s)\|$ and $\sup_{s \in [0,t]}\|X^\eps_s-x(s)-\eps Z_s\|$. 
We also note that the analysis can be extended to the case of nonperiodic impulses so long as one has some minimum time between impulses.

\begin{remark}\label{R:pks} 
We now take a moment to comment on translating equation \eqref{E:kicked-ode}, which involves forcing by a train of delta functions $\{\delta(t-t_k)\}_{k=1}^\infty$ with magnitude governed by $g(x)$, into the system \eqref{E:sie-det-per} which involves an {\sc ode} with impulses at times $\{t_k\}_{k=1}^\infty$ governed by the resetting map $h(x)$.
That the rigorous interpretation of \eqref{E:kicked-ode} should involve impulsive resetting at the times $\{t_k\}_{k=1}^\infty$ is hardly surprising.
However, to understand how to transition from the function $g$ in \eqref{E:kicked-ode} to $h$ in \eqref{E:sie-det-per} may be less obvious.
Intuitively, at least, one expects to use some \textit{regularization} of the delta function, solve the regularized {\sc ode}, and then let the regularization parameter tend to zero. 
In such situations, some subtleties may arise, given that there are multiple possible regularizations, and perhaps more importantly, various (or all) components of the state $x(t)$ may have a jump at $t_k$ prompting questions about whether to use the right limit or the left limit (or an average) when evaluating $g(x)$; see \cite{Balasuriya_Nonlinearity} for more details.
Of course, this concern does not arise when $g$ is independent of components of $x$ which experience a jump. 
More importantly, however, we note that our results are somewhat impervious to the exact interpretation chosen (to proceed from $g$ in \eqref{E:kicked-ode} to $h$ in \eqref{E:sie-det-per}) since we work directly with the function $h$, however the latter may have been obtained.
In light of this, we can, for concreteness, approximate the train of delta functions by the regularization $p^\gamma(t):[0,\infty) \to [0,\infty)$ given by
$p^\gamma(t) \triangleq \frac{1}{\gamma}\sum_{k \ge 1}\ind_{[t_k,t_k+\gamma]}(t)$ where $0<\gamma \ll 1$. 
Following \cite{WY_CMP2003}, we will say that a function $x:[0,\infty) \to \BR^d$ solves the initial value problem \eqref{E:kicked-ode} if 
\begin{enumerate}[(i)]
\item $x(0)=x_0$,
\item $\dot{x}(t)=b(x(t))$ for $t \neq t_k$, $k \ge 1$, and 
\item $x^+(t_k)=h(x^-(t_k))$ for $k \ge 1$, where the resetting map $h:\BR^d \to \BR^d$ is defined by
\begin{multline}\label{E:kick}
h(r) \triangleq \lim_{\gamma \searrow 0} \sfz^\gamma(\gamma;r) \quad \text{where $\sfz^\gamma(t;r)$ solves the initial value problem} \\\dot{\sfz}^\gamma(t;r) = \frac{1}{\gamma} g(\sfz^\gamma(t;r)) \quad \text{for $t>0$}, \quad \sfz^\gamma(0;r)=r \in \BR^d.
\end{multline}
\end{enumerate}
An important special case occurs when $g(x)=Ax+c$ where $A \in \BR^{d\times d}$ is a constant matrix, and $c \in \BR^d$ is a constant vector. It is easily seen that in this case, we have
$\sfz^\gamma(t;r) = e^{tA/\gamma}r + e^{tA/\gamma}\left(\int_0^t \frac{1}{\gamma}e^{-sA/\gamma} \thinspace ds\right) c$ which yields $h(r)=e^A r + e^A \left(\int_0^1 e^{-sA} \thinspace ds\right)c$, 
by virtue of \eqref{E:kick}. In particular, if $A=0$ (i.e., $g(x)=c$), then $h(r)=r+c$ which corresponds, as expected, to translation by amount $c$.
\end{remark}

We close out this section by stating a multidimensional version of Taylor's formula with remainder which will be used repeatedly in Sections \ref{S:LLN} and \ref{S:CLT} to prove Theorems \ref{T:LLN} and \ref{T:CLT}, respectively. 
Our notation here follows \cite{apostol1974mathematical}. For \( x, y \in \mathbb{R}^d \), let 
$L(x, y) \triangleq \{sx + (1-s)y : 0 \leq s \leq 1\}$
denote the line segment joining \( x \) and \( y \).
Suppose we have a function \( f : \mathbb{R}^d \to \mathbb{R} \) whose first and second order partial derivatives exist. Given \( x, t \in \mathbb{R}^d \), we let \( f'(x; t) \), \( f''(x; t) \) denote the quantities
\[
f'(x; t) \triangleq \sum_{i=1}^{d} \frac{\partial f}{\partial x_i}(x)t_i = \langle \nabla f(x), t \rangle, \quad \text{and} \quad 
f''(x; t) \triangleq \sum_{i=1}^{d} \sum_{j=1}^{d} \frac{\partial^2 f}{\partial x_i \partial x_j}(x)t_i t_j = \langle D^2 f(x)t, t \rangle,
\]
where $D^2 f(x) \triangleq \left[\frac{\partial^2 f}{\partial x_i \partial x_j}\right]_{1 \le i,j \le d}$ denotes the Hessian matrix and $\la \cdot,\cdot \ra$ denotes the standard inner product in $\BR^d$.
The symbol \( f^{(m)}(x; t) \) can be defined in a similar way if all \( m \)-th order partial derivatives exist; thus, $f^{(3)}(x; t) \triangleq \sum_{i=1}^{d} \sum_{j=1}^{d} \sum_{k=1}^{d} \frac{\partial^3 f}{\partial x_i \partial x_j \partial x_k}(x) t_i t_j t_k$, assuming the existence of third order partial derivatives.

\noindent \begin{proposition}[Taylor's formula] \cite{apostol1974mathematical} Let \( S \) be an open subset of \( \mathbb{R}^d \), and suppose we have a function \( f : S \to \mathbb{R} \). Assume that \( f \) and all its partial derivatives of order less than \( m \) are differentiable at each point of \( S \). If \( x, y \) are two points in \( S \) such that \( L(x, y) \subseteq S \), then there is a point \( z \) on the line segment \( L(x, y) \) such that
	\[
	f(y) - f(x) = \sum_{k=1}^{m-1} \frac{1}{k!} f^{(k)}(x; y - x) + \frac{1}{m!} f^{(m)}(z; y - x).
	\]
\end{proposition}


	\section{Limiting Mean Behavior}\label{S:LLN}
	In the present section, we provide the proof of Theorem \ref{T:LLN}, breaking the arguments into a series of propositions.	
	In the sequel, we will apply Taylor's formula to each component of the vector-valued functions $b,h:\BR^d \to \BR^d$. Writing this in vector notation will necessitate working with the Jacobian matrices $Db(\cdot)$, $Dh(\cdot)$, where the different rows may involve evaluating the gradient of that particular component at different spatial locations. To ease this process, we introduce the following notation.
	
	\begin{definition}\label{D:tf} 
		Let $f:\BR^d \to \BR^d$ be smooth. For $\bm \xi = (\xi_1,\xi_2,\dots,\xi_d)$ with $\xi_i \in \BR^d$ for $1 \le i \le d$, let $Df(\bm \xi)$ denote the $d \times d$ matrix whose $i$-th row is $\nabla f_i(\xi_i)$ for $1 \le i \le d$.
	\end{definition}

	

Recall the convention regarding matrix multiplication using the product notation as described in \eqref{E:matrix-mult-conv}.	
	\begin{lemma}\label{L:LLN-path-diff}
		For $\eps\in(0,1)$, $t\in[0, \sfT]$, and $\ren(t)$ as in \eqref{E:ren}, the solutions to equations \eqref{E:det-state-int-eq} and \eqref{E:stoch-state-int-eq} satisfy the relation
		\begin{equation}\label{E:R-error-ren-G}
			\begin{aligned}
				X^\eps_{t}-x(t) &= \prod_{j=1}^{\ren(t)}Dh(\bm\eta^{j,\eps})\left\{X^{\eps,+}_{0}-x^+(0)\right\} + \LL^\eps(t) + \MM^\eps(t), \quad \text{where} \\  
				\LL^\eps(t) &\triangleq  \sum_{i=1}^{\ren(t)} \left({\prod_{j=i}^{\ren(t)}} Dh(\bm\eta^{j,\eps})\right) \int_{t_{i-1}}^{t_i} Db(\bm\xi^{\eps}_s) \left(X^\eps_{s}-x(s)\right) \thinspace ds + \int_{t_{\ren(t)}}^t Db(\bm\xi^{\eps}_s)\left(X^\eps_{s}-x(s)\right) \thinspace ds\\
				\MM^\eps(t) &\triangleq  \eps \sum_{i=1}^{\ren(t)} \left({\prod_{j=i}^{\ren(t)}} Dh(\bm\eta^{j,\eps})\right) \int_{t_{i-1}}^{t_i} \sigma(X^\eps_s)\thinspace dW_s	+ \eps\int_{t_{\ren(t)}}^{t} \sigma(X^\eps_s)\thinspace dW_s
			\end{aligned}
		\end{equation}
		where $\bm\eta^{k,\varepsilon}=(\eta^{k,\varepsilon,1},\dots,\eta^{k,\varepsilon,d})$ with each $\eta^{k,\varepsilon,i}$ being a point between $x^-(t_{k})$ and $X^{\varepsilon,-}_{t_k}$; $\bm\xi^{\varepsilon}_t=(\xi^{\varepsilon,1}_t,\dots,\xi^{\varepsilon,d}_t)$ with each $\xi^{\varepsilon,i}_t$ being a point between $x(t)$ and $X^\varepsilon_{t}$, for $t\in [0,\infty)\setminus\{t_j\}_{j=1}^\infty$. 
	\end{lemma}
	
	\begin{proof}[Proof of Lemma \ref{L:LLN-path-diff}]
		Note that $[0,\sfT]$ can be written as the disjoint union $\left(\cup_{k=1}^\sfN [t_{k-1},t_k)\right)\cup [t_\sfN,\sfT]$. It easily follows from \eqref{E:stoch-state-int-eq} that for $t \in [0,\sfT]$, we have
		\begin{multline}\label{E:R-time-changed}
			X^\eps_t= \sum_{k=1}^{\sfN} \ind_{[t_{k-1},t_k)}(t)\left\{X^{\eps,+}_{t_{k-1}} + \int_{t_{k-1}}^t b(X^\eps_u) \thinspace du +\eps\int_{t_{k-1}}^t \sigma(X^\eps_u) \thinspace dW_u \right\}\\
			+ \ind_{[t_{\sfN},\sfT]}(t)\left\{X^{\eps,+}_{t_{\sfN}} + \int_{t_{\sfN}}^t b(X^\eps_u) \thinspace du +\eps\int_{t_{\sfN}}^t \sigma(X^\eps_u) \thinspace dW_u \right\}.
		\end{multline}	
		\noindent
		Now suppose $t \in [t_{k-1}, t_{k})$ for some $1 \le k \le \sfN$ or $t \in [t_{k-1}, \sfT]$ with $k=\sfN+1$. Then, we see from equations \eqref{E:det-state-int-eq} and \eqref{E:R-time-changed}  that $X^{\varepsilon}_t - x(t)
		= X^{\varepsilon,+}_{t_{k-1}} -x^+(t_{k-1}) + \int_{t_{k-1}}^{t}  b(X^{\varepsilon}_s) \ ds - \int_{t_{k-1}}^{t} b(x(s)) \ ds + \varepsilon \int_{t_{k-1}}^{t}\sigma(X^\eps_s)\ dW_s$. Recalling the notation of Definition \ref{D:tf} and using Taylor's formula, we get 
		\begin{multline}\label{E:R-error-k}
			X^{\varepsilon}_t - x(t) = \left\{X^{\varepsilon,+}_{t_{k-1}} - x^+(t_{k-1})\right\} + \II^\eps_k(t), \qquad \text{where} \\
			\II^\eps_k(t) \triangleq \int_{t_{k-1}}^{t} Db(\bm\xi^{\varepsilon}_s)\left(X^{\varepsilon}_{s}-x(s)\right) \ ds + \varepsilon\int_{t_{k-1}}^{t}  \sigma(X^\eps_s) \thinspace dW_s 
		\end{multline}
		where $\bm\xi^{\varepsilon}_s=(\xi^{\varepsilon,1}_s, \dots, \xi^{\varepsilon,d}_s)$ with each $\xi^{\varepsilon,i}_s$ being a point between $x(s)$ and $X^\varepsilon_{s}$. As $t \nearrow t_k$, $1 \le k \le \sfN$, the expression above approaches $X^{\eps,-}_{t_k}-x^-(t_k)$. Recalling the resetting rule in the last lines of \eqref{E:det-state-int-eq} and \eqref{E:stoch-state-int-eq}, we use Taylor's formula to get 
		\begin{equation}\label{E:R-resetting-k} 
			X^{\varepsilon,+}_{t_k} - x^+(t_{k})=h(X^{\varepsilon,-}_{t_k})-h(x^-(t_{k}))=Dh(\bm\eta^{k,\varepsilon})(X^{\varepsilon,-}_{t_k}-x^-(t_{k})), 
		\end{equation} 
		where $\bm\eta^{k,\varepsilon}=(\eta^{k,\varepsilon,1}, \dots, \eta^{k,\varepsilon,d})$ with each $\eta^{k,\varepsilon,i}$ being a point between $x^-(t_{k})$ and $X^{\varepsilon,-}_{t_k}$. Starting from the interval $[t_0,t_1)$ and working forward in time, successive alternate application of \eqref{E:R-error-k} followed by \eqref{E:R-resetting-k} yields that for $t \in [t_{k-1},t_k)$ with $1 \le k \le \sfN$, or $t \in [t_{k-1},\sfT]$ with $k=\sfN+1$, we have
		\begin{equation*}
			X^\eps_{t}-x(t)= {\prod_{j=1}^{k-1}}Dh(\bm\eta^{j,\eps})\left\{X^{\eps,+}_{t_0}-x^+(t_0)\right\} + \sum_{i=1}^{k-1} \left({\prod_{j=i}^{k-1}} Dh(\bm\eta^{j,\eps})\right) \II^\eps_i(t_i) + \II^\eps_k(t).
		\end{equation*}
		
		Unwrapping this expression and suitably regrouping the terms, we see that for $t \in [t_{k-1},t_k)$ with $1 \le k \le \sfN$ or $t \in [t_{k-1},\sfT]$ with $k=\sfN+1$, we have
		\begin{equation}\label{E:R-error-G}
			\begin{aligned}
				X^\eps_{t}-x(t) &= {\prod_{j=1}^{k-1}}Dh(\bm\eta^{j,\eps})\left\{X^{\eps,+}_{t_0}-x^+(t_0)\right\} + \LL^\eps_k(t) + \MM^\eps_k(t), \quad \text{where} \\  
				\LL^\eps_k(t) &\triangleq  \sum_{i=1}^{k-1} \left({\prod_{j=i}^{k-1}} Dh(\bm\eta^{j,\eps})\right) \int_{t_{i-1}}^{t_i} Db(\bm\xi^{\eps}_s) \left(X^\eps_{s}-x(s)\right) \thinspace ds + \int_{t_{k-1}}^t Db(\bm\xi^{\eps}_s)\left(X^\eps_{s}-x(s)\right) \thinspace ds\\
				\MM^\eps_k(t) &\triangleq  \eps \sum_{i=1}^{k-1} \left({\prod_{j=i}^{k-1}} Dh(\bm\eta^{j,\eps})\right) \int_{t_{i-1}}^{t_i} \sigma(X^\eps_s)\thinspace dW_s
				+ \eps\int_{t_{k-1}}^{t} \sigma(X^\eps_s)\thinspace dW_s.
			\end{aligned}
		\end{equation}
		Observing that $\ren(t)=k-1$ in equation \eqref{E:R-error-G}, we get the stated result.
	\end{proof}


We will frequently use the fact, easily verified by the triangle inequality, that for any $0 \le s_1 \le s_2 \le t$, we have
\begin{equation}\label{E:stoch-int-est}
\left\|\int_{s_1}^{s_2} \sigma(X^\eps_u) \thinspace dW_u \right\| \le 2 \sup_{0 \le s \le t} \left\|\int_{0}^{s} \sigma(X^\eps_u) \thinspace dW_u \right\|.
\end{equation}	
	
\begin{lemma}\label{L:LLN-comparison}
There exists a constant $C_{\ref{L:LLN-comparison}}>0$ such that for all $\eps \in (0,1)$, $t\in[0, \sfT]$, we have		 
\begin{equation*}
\sup_{0 \le s \le t}\|X^\eps_s - x(s)\| \le C_{\ref{L:LLN-comparison}} \left\{ \int_0^t \sup_{0 \le u \le s}\|X^\eps_{u}-x(u)\| \thinspace ds + \eps \sup_{0 \le s \le t} \left\|\int_{0}^{s} \sigma(X^\eps_u) \thinspace dW_u \right\|\right\}.
\end{equation*}
\end{lemma}

\begin{proof}[Proof of Lemma \ref{L:LLN-comparison}]
We now carefully estimate the terms on the right hand side of the first line in \eqref{E:R-error-ren-G}. The first term is zero on account of the initial conditions. It is easily checked, using Assumptions \ref{A:vf} and \ref{A:resetting-map}, that $\|Db(\bm\xi^\eps_s)\|_\sfF \le \sfK_b d$, $\|Dh(\bm\eta^{j,\eps})\|_\sfF \le \sfK_h d$. We now have
\begin{equation*}
\|\LL^\eps(t)\| \le \sum_{i=1}^{\ren(t)} \left(\prod_{j=i}^{\ren(t)} \|Dh(\bm\eta^{j,\eps})\|_\sfF\right) \int_{t_{i-1}}^{t_i} \|Db(\bm\xi^\eps_s)\|_\sfF \|X^\eps_s-x(s)\|\thinspace ds + \int_{t_{\ren(t)}}^t \|Db(\bm\xi^\eps_s)\|_\sfF \|X^\eps_s-x(s)\|\thinspace ds.
\end{equation*}
Thus, $\|\LL^\eps(t)\| \le \sum_{i=1}^{\ren(t)} (\sfK_h d)^{\ren(t)} (\sfK_b d)\int_{t_{i-1}}^{t_i} \sup_{0 \le u \le s}\|X^\eps_u-x(u)\| \thinspace ds + \sfK_b d \int_{t_{\ren(t)}}^t \sup_{0 \le u \le s}\|X^\eps_u-x(u)\| \thinspace ds$. Noting that $\ren(t)$ is nondecreasing in $t$ with $\ren(\sfT)=\sfN<\infty$, there exists a constant $C_1>0$ such that 
\begin{equation}\label{E:L-estimate-LLN}
\|\LL^\eps(t)\| \le C_1 \int_0^t \sup_{0 \le u \le s} \|X^\eps_u-x(u)\| \thinspace ds \quad \text{for all $t \in [0,\sfT]$.}
\end{equation}
Turning to $\MM^\eps(t)$, we note that
\begin{equation*}
\|\MM^\eps(t)\| \le \eps \sum_{i=1}^{\ren(t)} \left(\prod_{j=i}^{\ren(t)} \|Dh(\bm\eta^{j,\eps})\|_\sfF\right) \left\|\int_{t_{i-1}}^{t_i} \sigma(X^\eps_s) \thinspace dW_s\right\| + \left\|\int_{t_{\ren(t)}}^t \sigma(X^\eps_s) \thinspace dW_s\right\|.
\end{equation*}
Recalling \eqref{E:stoch-int-est} and reasoning as above, there exists a constant $C_2 > 0$ such that
\begin{equation}\label{E:M-estimate-LLN}
\|\MM^\eps(t)\| \le \eps C_2 \sup_{0 \le s \le t} \left\|\int_{0}^{s} \sigma(X^\eps_u) \thinspace dW_u \right\|.
\end{equation}
Using \eqref{E:L-estimate-LLN} in conjunction with equations \eqref{E:L-estimate-LLN} and \eqref{E:M-estimate-LLN}, and noting that the right-hand sides for the latter two equations are nondecreasing in $t$, we easily get the stated result.
\end{proof}

%
%
%
%
%
%

The next lemma allows us to estimate $\sup_{0 \le s \le t}\|X^\eps_s - x(s)\|^2$ and will be used both in the proof of Theorem \ref{T:LLN} and \ref{T:CLT}.

\begin{lemma}\label{L:LLN-comparison-squared} 
Fix $\sfT>0$. There exists a constant $C_{\ref{L:LLN-comparison-squared}}>0$ such that for all $\eps \in (0,1)$, $t\in[0, \sfT]$, we have 
\begin{equation}\label{E:LLN-comparison-squared}
\sup_{0 \le s \le t}\|X^\eps_s - x(s)\|^2
		\le  C_{\ref{L:LLN-comparison-squared}} \left\{\int_0^t \sup_{0 \le u \le s}\|X^\eps_{u}-x(u)\|^2 \thinspace ds + \eps^2 \sup_{0\le s\le t}\left\|\int_{0}^{s} \sigma(X^\eps_u)\thinspace dW_u\right\|^2\right\}.
\end{equation} 
\end{lemma}

\begin{proof}[Proof of Lemma \ref{L:LLN-comparison-squared}]
Noting that for each $\omega \in \Omega$, $s \mapsto \sup_{0 \le u \le s}\|X^\eps_u-x(u)\|$ is square-integrable on $[0,t]$ with respect to Lebesgue measure, an application of Holder's inequality (upon squaring) yields\\ $\left(\int_0^t \sup_{0 \le u \le s}\|X^\eps_u-x(u)\| \thinspace ds\right)^2 \le t \int_0^t \sup_{0 \le u \le s}\|X^\eps_u-x(u)\|^2 \thinspace ds$. Using the inequality $(a+b)^2 \le 2(a^2+b^2)$ for all $a,b \in \BR$, the stated claim now follows from Lemma \ref{L:LLN-comparison}.
\end{proof}

We next estimate the mean contribution of the stochastic integrals.

\begin{lemma}\label{L:Diff_coeff}
For $\eps \in (0,1)$, $t\in[0, \sfT]$,  we have
\begin{equation}\label{E:Diff_coeff}
\BE\left[\sup_{0 \le s \le t} \left\|\int_0^s \sigma(X^\eps_u) \thinspace dW_u\right\|\right] \le \left\{\BE\left[\sup_{0 \le s \le t} \left\|\int_0^s \sigma(X^\eps_u) \thinspace dW_u\right\|^2\right]\right\}^{1/2}  \le 2r \thinspace \left\{\mathbb{E}\left[\int_{0}^{t} \|\sigma(X^\eps_s)\|^2_\sfF\thinspace ds\right]\right\}^{1/2}.
\end{equation}
\end{lemma}

\begin{proof}[Proof of Lemma \ref{L:Diff_coeff}]
The first inequality in \eqref{E:Diff_coeff} follows from monotonicity of $L^p$ norms on a probability space. To estimate the quantity $\mathbb{E}\left[\sup_{0\le s \le t} \left\|\int_{0}^{s}\sigma(X^\eps_u)\thinspace dW_u\right\|^2\right]$, we start by noting that
$\left\|\int_{0}^{s}\sigma(X^\eps_u)\thinspace dW_u\right\|^2 = \sum_{i=1}^d \left(\sum_{j=1}^r \int_0^s \sigma_{ij}(X^\eps_u) \thinspace dW^j_u\right)^2 \le  \sum_{i=1}^d \left(\sum_{j=1}^r \left|\int_0^s \sigma_{ij}(X^\eps_u) \thinspace dW^j_u\right|\right)^2 
\le r^2 \sum_{i=1}^d \sum_{j=1}^r \left|\int_0^s \sigma_{ij}(X^\eps_u) \thinspace dW^j_u\right|^2$. Taking suprema over $[0,t]$, we get
\begin{equation}\label{E:sup-est}
\sup_{0\le s \le t} \left\|\int_{0}^{s}\sigma(X^\eps_u)\thinspace dW_u\right\|^2 \le r^2 \sum_{i=1}^d \sum_{j=1}^r \sup_{0 \le s \le t} \left|\int_0^s \sigma_{ij}(X^\eps_u) \thinspace dW^j_u\right|^2.
\end{equation}
By Doob's maximal inequality applied to the continuous submartingale $\left|\int_{0}^{s} \sigma_{ij}(X^\eps_u)\thinspace dW^j_u\right|$, we have\\ $\mathbb{E}\left[ \sup_{0\le s \le t }\left|\int_{0}^{s} \sigma_{ij}(X^\eps_u)\thinspace dW^j_u\right|^2\right]\le 4 \mathbb{E}\left(\left|\int_{0}^{t} \sigma_{ij}(X^\eps_s) \thinspace dW^j_s\right|^2\right)$. Taking expectations in \eqref{E:sup-est} and using the Ito isometry, we get
\begin{equation*}
\BE\left[\sup_{0\le s \le t} \left\|\int_{0}^{s}\sigma(X^\eps_u)\thinspace dW_u\right\|^2\right] \le 4r^2 \sum_{i=1}^d \sum_{j=1}^r\BE\left[\int_0^t  \sigma_{ij}(X^\eps_s)^2 \thinspace ds\right]=4r^2 \BE\left[\int_0^t \|\sigma(X^\eps_s)\|_\sfF^2 \thinspace ds\right],
\end{equation*}
which immediately yields the stated result.
\end{proof}

Next we wish to present an upper bound of $\mathbb{E}\left(\int_0^t \|\sigma(X^\eps_s)\|^2_F\thinspace ds\right)$ and $\mathbb{E}\left(\int_0^t |X^\eps_s|^2\thinspace ds\right)$ , which will be instrumental in the next proposition.

\begin{lemma}\label{L:stoch-traj-UB}
There exists a constant $C_{\ref{L:stoch-traj-UB}}>0$ such that for $\eps\in(0, 1)$ and $t\in[0, \sfT]$, we have
\begin{equation}
\mathbb{E}\left(\int_{0}^{t}\sup_{0\le u\le s}\|X^\eps_u\|^2 \thinspace ds \right)  
\le C_{\ref{L:stoch-traj-UB}} e^{C_{\ref{L:stoch-traj-UB}}t}, \quad \text{and} \quad 
\mathbb{E}\left(\int_0^t \|\sigma(X^\eps_s)\|^2_\sfF\thinspace ds\right)\le C_{\ref{L:stoch-traj-UB}} e^{C_{\ref{L:stoch-traj-UB}} t}.
\end{equation}
\end{lemma}

\begin{proof}[Proof of Lemma \ref{L:stoch-traj-UB}]
For the concatenated path traced by $X^\eps_t$, we start by obtaining a decomposition which resembles that in \eqref{E:R-error-ren-G}. The calculations proceed along very similar lines, except that at each impulse time $t_k$, $k \ge 1$, Taylor's formula takes the form 
\begin{equation*}
X^{\eps,+}_{t_k} = h(X^{\eps,-}_{t_k}) = h(0) + Dh(\bm \zeta^{k,\eps})X^{\eps,-}_{t_k} \qquad \text{where $\bm \zeta^{k,\eps} = (\zeta^{k,\eps,1}, \dots, \zeta^{k,\eps,d})$}
\end{equation*} 
with $\zeta^{k,\eps,i}$, $1 \le i \le d$, being a point between $0$ and $X^{\eps,-}_{t_k}$. This yields
\begin{multline}\label{E:X-B}
X^\eps_t =  \sum_{i=1}^{\ren(t)}\left({\prod_{j=i+1}^{\ren(t)}} Dh(\bm \zeta^{j,\eps})\right)h(0) + \left({\prod_{j=1}^{\ren(t)}} Dh(\bm \zeta^{j,\eps})\right)X^{\eps,+}_{t_0} \\
+ \sum_{i=1}^{\ren(t)} \left({\prod_{j=i}^{\ren(t)}} Dh(\bm \zeta^{j,\eps})\right) \int_{t_{i-1}}^{t_i} b(X^\eps_s) \thinspace ds 
+ \int_{t_{\ren(t)}}^t b(X^\eps_s) \thinspace ds \\
+ \eps \sum_{i=1}^{\ren(t)} \left({\prod_{j=i}^{\ren(t)}} Dh(\bm \zeta^{j,\eps})\right) \int_{t_{i-1}}^{t_i} \sigma(X^\eps_s)\thinspace dW_s
+ \eps\int_{t_{\ren(t)}}^{t} \sigma(X^\eps_s)\thinspace dW_s.
\end{multline}
As noted earlier, Assumptions \ref{A:vf} and \ref{A:resetting-map} imply that $\|Dh(\bm\zeta^{j,\eps})\|_\sfF \le \sfK_h d$. Since $\ren(t)$ is non-decreasing in $t$ with $\ren(\sfT)=\sfN<\infty$, a simple use of the triangle inequality in \eqref{E:X-B} yields
\begin{multline}\label{E:X-BB}
\|X^\eps_t\|\le \sfN (\sfK_h d)^{\sfN} \left\{ \|h(0)\| + \|x_0\| + \int_0^t \|b(X^\eps_s)\| \thinspace ds \right.\\
\left.+ \eps \sum_{i=1}^{\ren(t)} \left\|\int_{t_{i-1}}^{t_i} \sigma(X^\eps_s) \thinspace dW_s \right\| + \eps \left\|\int_{t_{\ren(t)}}^t \sigma(X^\eps_s) \thinspace dW_s\right\| \right\}.
\end{multline}	
Squaring \eqref{E:X-BB} and using H\"older's inequality, it is easily seen that there exists a constant $\tilde{C}>0$ such that for all $t \in [0,\sfT]$, we have
\begin{multline*}
\|X^\eps_t\|^2 \le \tilde C\left\{\|h(0)\|^2 + \|x_0\|^2 + \int_0^t \|b(X^\eps_s)\|^2 \thinspace ds + \eps^2 \sum_{i=1}^{\ren(t)} \left\|\int_{t_{i-1}}^{t_i} \sigma(X^\eps_s) \thinspace dW_s\right\|^2 \right.\\ 
\left.+ \eps^2 \left\|\int_{t_{\ren(t)}}^t \sigma(X^\eps_s) \thinspace dW_s\right\|^2\right\}\\
\le \tilde C\left\{\|h(0)\|^2 + \|x_0\|^2 + \int_0^t \|b(X^\eps_s)\|^2 \thinspace ds + 4\eps^2 (\ren(\sfT)+1) \sup_{0 \le s \le t}\left\|\int_{0}^s \sigma(X^\eps_u) \thinspace dW_u\right\|^2 \right\},
\end{multline*}
where the latter follows from \eqref{E:stoch-int-est}. Note that by Assumptions \ref{A:vf} and \ref{A:resetting-map}, there exists a constant $K>0$ such that $\|b(x)\|^2 +\|\sigma(x)\|_\sfF^2 \le K (1+\|x\|^2)$ for all $x\in \mathbb{R}^d$. Redefining $\tilde C$, we easily get 
\begin{equation}\label{E:X-squared}
\sup_{0 \le s \le t} \|X^\eps_s\|^2 \le \tilde C\left\{1 + \|x_0\|^2 + \int_0^t \sup_{0 \le u \le s}\|X^\eps_u\|^2 \thinspace ds + \eps^2 \sup_{0 \le s \le t}\left\|\int_{0}^s \sigma(X^\eps_u) \thinspace dW_u\right\|^2 \right\}.
\end{equation}
Taking expectations in \eqref{E:X-squared}, recalling Lemma \ref{L:Diff_coeff}, and using the aforementioned linear growth of $\sigma$, we easily see that there exists a constant $C>0$ (depending on $\|x_0\|$) such that 
\begin{equation*}
\BE\left[\sup_{0 \le s \le t} \|X^\eps_s\|^2\right] \le C\left\{1 + \int_0^t \BE\left[\sup_{0 \le u \le s}\|X^\eps_u\|^2\right] \thinspace ds\right\}  
\end{equation*}	
Gronwall's inequality implies that for all $t \in [0,\sfT]$, $\eps \in (0,1)$, we have
\begin{equation*}
\BE\left[\sup_{0 \le s \le t} \|X^\eps_s\|^2\right] \le Ce^{Ct} \quad \text{and upon integrating,} \quad \BE\left[\int_0^t \sup_{0 \le u \le s}\|X^\eps_u\|^2 \thinspace ds \right] \le e^{Ct},
\end{equation*}
thereby proving the first claim in equation \eqref{L:stoch-traj-UB}. A simple calculation using the linear growth of $\sigma$ easily yields the second claim in \eqref{L:stoch-traj-UB}.
\end{proof}

We now provide the proof of Theorem \ref{T:LLN}.

\begin{proof}[Proof of Theorem \ref{T:LLN}]
Consider the case $p=2$. Taking expectations in \eqref{E:LLN-comparison-squared}, using the second inequality in \eqref{E:Diff_coeff} to estimate $\BE\left[\sup_{0 \le s \le t}\|\int_0^s \sigma(X^\eps_u)\thinspace dW_u\|^2\right]$ in terms of $\BE\left[\int_0^t \|\sigma(X^\eps_s)\|_\sfF^2 \thinspace ds\right]$, and estimating the latter using Lemma \ref{L:stoch-traj-UB}, we see that there exists a constant $C>0$ such that
\begin{equation*}
\BE\left[\sup_{0 \le s \le t}\|X^\eps_s-x(s)\|^2\right] \le C\left\{\eps^2 + \int_0^t \BE\left[\sup_{0 \le u \le s}\|X^\eps_u-x(u)\|^2\right] \thinspace ds\right\} \quad \text{for all $t \in [0,\sfT]$, $\eps \in (0,1)$.}
\end{equation*}
Gronwall's inequality now yields $\BE\left[\sup_{0 \le s \le t}\|X^\eps_s-x(s)\|^2\right] \le \eps^2 C e^{C\sfT}$, thereby proving \eqref{E:LLN-rate} for the case $p=2$. By monotonicity of $L^p$ norms on a probability space, we have\\ $\BE\left[\sup_{0 \le s \le t}\|X^\eps_s-x(s)\|\right] \le \left(\BE\left[\sup_{0 \le s \le t}\|X^\eps_s-x(s)\|\right]\right)^{1/2}$ which immediately yields \eqref{E:LLN-rate} for the case $p=1$.
\end{proof}



\section{Analysis of Fluctuations}\label{S:CLT}
Our main goal in this section is to prove Theorem \ref{T:CLT}. 
Since we intend to apply the Taylor formula to vector-valued functions $f:\BR^d \to \BR^d$, we will find the following notation useful. Suppose $\bm \xi = (\xi_1,\dots,\xi_d)$ where each $\xi_i \in \BR^d$ for $1 \le i \le d$. For $x \in \BR^d$, let $\la D^2 f(\bm \xi)x,x \ra$ denote the column vector in $\BR^d$ with $i$-th component $\la D^2 f_i(\xi_i)x,x \ra$, i.e., 
\begin{equation}\label{E:D-squared}
\la D^2 f(\bm \xi)x,x \ra \triangleq \left(\la D^2 f_1(\xi_1)x,x \ra, \la D^2 f_2(\xi_2)x,x \ra,\dots,\la D^2 f_d(\xi_d)x,x \ra\right)^\intercal
\end{equation}


The next lemma gives a comparison between the path traced by $X^\eps$ and $x+\eps Z$. 


\begin{lemma}\label{L:CLT-path-diff}
For any $\eps\in(0,1)$ and for $t\in[0, \sfT]$, $\ren(t)$ defined as above, the solutions to equations \eqref{E:det-state-int-eq}, \eqref{E:stoch-state-int-eq}, and \eqref{E:fluct-proc} satisfy the following relation
\begin{multline}\label{E:CLT-decomp-detailed}
X^\eps_{t}-x(t) - \eps Z_t =\prod_{j=1}^{\ren(t)}Dh(x^-(t_j))(X^{\eps,+}_{t_0}-x^+(t_0) - \eps Z^{+}_{t_0}) 
+ \hat\JJ^\eps(t) 
+ \hat{\LL}^\eps(t) + \hat{\MM}^\eps(t) + \hat{\NN}^\eps(t) 	
\end{multline} 
where
\begin{equation}\label{E:JLMN}
\begin{aligned}
\hat\JJ^\eps(t) &\triangleq \frac{1}{2}\sum_{i=1}^{\ren(t)} \left(\prod_{j=i+1}^{\ren(t)} Dh(x^-(t_j))\right) \color{black} \la D^2 h(\hat{\bm\eta}^{i,\varepsilon})(X^{\eps, -}_{t_i} - x^{-}(t_i)),(X^{\eps, -}_{t_i} - x^{-}(t_i)) \ra \\
\hat\LL^\eps(t) &\triangleq \frac{1}{2} \sum_{i=1}^{\ren(t)} \left(\prod_{j=i}^{\ren(t)} Dh(x^-(t_{j}))\right) \int_{t_{i-1}}^{t_i} \la D^2 b(\hat{\bm\xi}^{\varepsilon}_s) (X^{\varepsilon}_{s}-x(s)),(X^{\varepsilon}_{s}-x(s)) \ra \thinspace ds \\
& \qquad + \frac{1}{2}\int_{t_{\ren(t)}}^{t} \la D^2 b(\hat{\bm\xi}^{\varepsilon}_s)(X^{\varepsilon}_{s}-x(s)), (X^{\varepsilon}_{s}-x(s)) \ra\thinspace ds \\
\hat\MM^\eps (t) &\triangleq  \sum_{i=1}^{\ren(t)} \left(\prod_{j=i}^{\ren(t)} Dh(x^-(t_{j}))\right) \int_{t_{i-1}}^{t_i} Db(x(s))\left(X^\eps_{s}-x(s)-\eps Z_s\right) \thinspace ds \\ 
&+\int_{t_{\ren(t)}}^t Db(x(s))\left(X^\eps_{s}-x(s)-\eps Z_s\right) \thinspace ds
\\
\hat\NN^\eps (t) &\triangleq  \eps \sum_{i=1}^{\ren(t)} \left(\prod_{j=i}^{\ren(t)} Dh(x^-(t_{j}))\right) \int_{t_{i-1}}^{t_i} \left(\sigma(X^\eps_s) -\sigma(x(s))\right) \ dW_s +\eps \int_{t_{\ren(t)}}^{t} \left(\sigma(X^\eps_s) -\sigma(x(s))\right) \ dW_s 
		\end{aligned}
	\end{equation}
where $\hat{\bm\eta}^{i,\varepsilon}=(\hat{\eta}^{i,\varepsilon,1},\dots,\hat{\eta}^{i,\varepsilon,d})$ with each $\hat{\eta}^{i,\varepsilon,j}$ being a point between $x^-(t_{i})$ and $X^{\varepsilon,-}_{t_i}$, and $\hat {\bm \xi}^{\eps}_s \triangleq (\hat\xi^{\eps,1}_s,\dots,\hat\xi^{\eps,d}_s)$ where each $\hat\xi^{\eps,j}_s$, $1 \le j \le d$, is a point between $x(s)$ and $X^\varepsilon_{s}$ for $s \in [0,\sfT]\setminus\{t_k\}_{k=1}^\sfN$. 
\end{lemma}


\begin{proof}[Proof of Lemma \ref{L:CLT-path-diff}]
As in the proof of Lemma \ref{L:LLN-path-diff}, we start by writing $[0,\sfT]$ as the disjoint union $\left(\cup_{k=1}^\sfN [t_{k-1},t_k)\right)\cup [t_\sfN,\sfT]$. We also recall the decomposition for $X^\eps_t$ over the time horizon $[0,\sfT]$ given by \eqref{E:R-time-changed}, and note that similar expressions can be written for $x(t)$ and $Z_t$. Now, let $t \in [t_{k-1}, t_{k})$ for some $1 \le k \le \sfN$ or $t \in [t_{k-1}, \sfT]$ with $k=\sfN+1$. Using equations \eqref{E:det-state-int-eq}, \eqref{E:fluct-proc}, and \eqref{E:R-time-changed}, and adding and subtracting $\int_{t_{k-1}}^{t} Db(x(s)) \left(X^\eps_{s} - x(s)\right) \thinspace ds$ on the right-hand side, we have
\begin{multline*}
X^{\varepsilon}_{t} - x(t)- \eps Z_t
= \left(X^{\varepsilon,+}_{t_{k-1}} - x^+(t_{k-1})-\eps Z^{+}_{t_{k-1}}\right) + \int_{t_{k-1}}^{t}  \left[b(X^{\varepsilon}_{s}) - b(x(s)) - Db(x(s))(X^\eps_s-x(s))\right]\ ds \\+ \int_{t_{k-1}}^{t} Db(x(s))  \left[X^\eps_s - x(s) - \eps Z_s\right]\ ds + \varepsilon \int_{t_{k-1}}^{t} \left[\sigma(X^\eps_s) - \sigma(x(s))\right] \thinspace dW_s 
\end{multline*}
Using Taylor's formula, and recalling the notation in \eqref{E:D-squared}, there exists $\hat {\bm \xi}^{\eps}_s \triangleq (\hat\xi^{\eps,1}_s,\dots,\hat\xi^{\eps,d}_s)$ where each $\hat\xi^{\eps,j}_s$, $1 \le j \le d$, is a point between $x(s)$ and $X^\varepsilon_{s}$ such that 
\begin{multline}\label{E:R-error-k-clt}
X^{\varepsilon}_{t} - x(t) - \eps Z_t = \left(X^{\varepsilon,+}_{t_{k-1}} - x^+(t_{k-1}) -\eps Z^{+}_{t_{k-1}}\right) + \hat{\II}^\eps_k(t), \qquad \text{where} \\
\hat{\II}^\eps_k(t) \triangleq \frac{1}{2} \int_{t_{k-1}}^{t} \la D^2 b(\hat{\bm\xi}^{\varepsilon}_s)\left(X^{\varepsilon}_{s}-x(s)\right), \left(X^{\varepsilon}_{s}-x(s)\right)\ra\thinspace ds + \int_{t_{k-1}}^{t} Db(x(s))\left(X^\eps_{s} - x(s)-\eps Z_s\right)\ ds \\ +\varepsilon \int_{t_{k-1}}^{t} \left[\sigma(X^\eps_s) - \sigma(x(s))\right] \thinspace dW_s 
\end{multline}
As $t \nearrow t_k$, $1 \le k \le \sfN$, the expression above approaches $X^{\eps,-}_{t_k}-x^-(t_k) - \eps Z^{-}_{t_k}$. The resetting rules in \eqref{E:det-state-int-eq}, \eqref{E:stoch-state-int-eq} and \eqref{E:fluct-proc} imply that $X^{\varepsilon,+}_{t_k} - x^+(t_{k})-\eps Z^{+}_{t_k} =h(X^{\varepsilon,-}_{t_k})-h(x^-(t_{k}))- \eps Dh(x^{-}(t_{k})) Z^{-}_{t_k}$. If we now add and subtract $Dh(x^-(t_k))(X^{\eps,-}_{t_k}-x^-(t_k))$, and then use Taylor's formula, we get
\begin{multline}\label{E:R-resetting-k-clt}
X^{\varepsilon,+}_{t_k} - x^+(t_{k})-\eps Z^{+}_{t_k}
=\frac{1}{2} \la D^2 h(\hat{\bm\eta}^{k,\varepsilon})(X^{\varepsilon,-}_{t_k}-x^{-}(t_k)), (X^{\varepsilon,-}_{t_k}-x^{-}(t_k))\ra \\+ Dh(x^-(t_{k}))(X^{\eps, -}_{t_k} - x^-(t_k)-\eps Z^{-}_{t_k}), 
\end{multline} 
where $\hat{\bm\eta}^{k,\varepsilon}=(\hat{\eta}^{k,\varepsilon,1},\dots,\hat{\eta}^{k,\varepsilon,d})$ with each $\hat{\eta}^{k,\varepsilon,j}$ being a point between $x^-(t_{k})$ and $X^{\varepsilon,-}_{t_k}$. We now proceed in a manner similar to that in Lemma \ref{L:LLN-path-diff}. Starting from $[t_0,t_1)$, we work our way forward in time alternately applying \eqref{E:R-error-k-clt} and  \eqref{E:R-resetting-k-clt} in succession. As a result, we obtain that for $t \in [t_{k-1},t_k)$ with $1 \le k \le \sfN$, or $t \in [t_{k-1},\sfT]$ with $k=\sfN+1$, we have
	\begin{multline}\label{E:CLT-decomp-basic}
		X^\eps_{t}-x(t) - \eps Z_t = \prod_{j=1}^{k-1}Dh(x^-(t_j))\left(X^{\eps,+}_{t_0}-x^+(t_0) - \eps Z^{+}_{t_0}\right) \\+ \frac{1}{2}\sum_{i=1}^{k-1}  \left(\prod_{j=i+1}^{k-1} Dh(x^-(t_j))\right) \color{black} \la D^2 h(\hat{\bm\eta}^{i,\varepsilon})(X^{\varepsilon,-}_{t_i}-x^{-}(t_i)), (X^{\varepsilon,-}_{t_i}-x^{-}(t_i)) \ra
		\\+ \sum_{i=1}^{k-1}\left(\prod_{j=i}^{k-1} Dh(x^-(t_j))\right) \hat{\II}^\eps_i(t_i) + \hat{\II}^\eps_k(t)
	\end{multline}
	where the quantities $\hat{\II}^\eps_i(t)$ are as in equation \eqref{E:R-error-k-clt}, and we use the conventions $\sum_{i=1}^0 (\dots) \triangleq \bm 0$ and $\prod_{j=1}^0 (\dots) \triangleq I_d$ with $\bm 0$ and $I_d$ denoting the $d \times d$ zero and identity matrices respectively. 
Carefully expanding the right-hand side of \eqref{E:CLT-decomp-basic} using the expressions for $\hat\II^\eps_i(t)$ from \eqref{E:R-error-k-clt}, we get
\begin{multline}\label{E:CLT-decomp-detailed-O}
X^\eps_{t}-x(t) - \eps Z_t =\prod_{j=1}^{k-1}Dh(x^-(t_j))(X^{\eps,+}_{t_0}-x^+(t_0) - \eps Z^{+}_{t_0}) 
+ \hat\JJ^\eps_k(t)
 + \hat{\LL}^\eps_k(t) + \hat{\MM}^\eps_k(t) + \hat{\NN}^\eps_k(t) 
\end{multline} 
where	
	
\begin{equation}\label{E:JLMNO}
\begin{aligned}
\hat\JJ^\eps_k(t) &\triangleq \frac{1}{2}\sum_{i=1}^{k-1} \left(\prod_{j=i+1}^{k-1} Dh(x^-(t_j))\right) \la D^2 h(\hat{\bm\eta}^{i,\varepsilon})(X^{\varepsilon,-}_{i}-x^{-}(t_i)), (X^{\varepsilon,-}_{i}-x^{-}(t_i))\ra\\
\hat\LL^\eps_k(t) &\triangleq \frac{1}{2} \sum_{i=1}^{k-1} \left(\prod_{j=i}^{k-1} Dh(x^-(t_{j}))\right) \int_{t_{i-1}}^{t_i} \la D^2 b(\hat{\bm\xi}^{\eps}_s)(X^\eps_s -x(s)),(X^\eps_s -x(s)) \ra \thinspace ds \\
&\qquad + \frac{1}{2}\int_{t_{k-1}}^{t} \la D^2b(\hat{\bm\xi}^{\eps}_s)(X^\eps_s -x(s)), (X^\eps_s -x(s)) \ra \thinspace ds \\
\hat\MM^\eps_k(t) &\triangleq  \sum_{i=1}^{k-1} \left(\prod_{j=i}^{k-1} Dh(x^-(t_{j}))\right) \int_{t_{i-1}}^{t_i} Db(x(s))\left(X^\eps_{s}-x(s)-\eps Z_s\right) \thinspace ds 
\\ &+\int_{t_{k-1}}^t Db(x(s))\left(X^\eps_{s}-x(s)-\eps Z_s\right) \thinspace ds
\\
\hat\NN^\eps_k(t) &\triangleq  \eps \sum_{i=1}^{k-1} \left(\prod_{j=i}^{k-1} Dh(x^-(t_{j}))\right) \int_{t_{i-1}}^{t_i} \left(\sigma(X^\eps_s) -\sigma(x(s))\right) \ dW_s +\eps \int_{t_{k-1}}^{t} \left(\sigma(X^\eps_s) -\sigma(x(s))\right) \ dW_s .
\end{aligned}
\end{equation}
Writing equations \eqref{E:CLT-decomp-detailed-O} and \eqref{E:JLMNO} in terms of the number $\ren(t)$ of impulses up to time $t$, we get the stated result.
\end{proof}

\begin{lemma}\label{L:CLT-comparison}
There exists a constant $C_{\ref{L:CLT-comparison}}>0$ such that for all $\eps \in (0,1)$, $t \in [0,\sfT]$, we have
\begin{multline}\label{E:CLT-comparison}
\sup_{0 \le s \le t}\|X^\eps_{s}-x(s)- \eps Z_s\|  \le C_{\ref{L:CLT-comparison}}\left\{ \sup_{0 \le s \le t}\|X^{\varepsilon}_s-x(s)\|^2 
			+ \int_{0}^{t}  \sup_{0\le u\le s}\|X^{\varepsilon}_{u}-x(u)\|^2 \thinspace ds\right. \\  \qquad \left. + \int_0^t \sup_{0 \le u \le s}\|X^\eps_{u}-x(u)-\eps Z_u\| \thinspace ds
			+ \eps \sup_{0\le v \le t}\left\|\int_{0}^{v} (\sigma(X^\eps_u)- \sigma(x(u)))\thinspace dW_u\right\|\right\}. 
\end{multline}
\end{lemma}

\begin{proof}[Proof of Lemma \eqref{L:CLT-comparison}]
We start by working our way through the terms on the right-hand side of equation \eqref{E:CLT-decomp-detailed}. Note that the first term is zero on account of the initial conditions. Next, we note that from Assumption \ref{A:resetting-map}, we have $\|Dh(x)\|_\sfF \le \sfK_h d$, $\| D^2 h_k(x)\|_\sfF \le \sfK_h d$ for all $x \in \BR^d$, $1 \le k \le d$. Also, straightforward calculations using the Cauchy-Schwarz inequality yield $\|\la D^2 h(\bm\eta)x,x\ra\| \le \sfK_h d^2 \|x\|^2$ for all $x \in \BR^d$ and $\bm\eta = (\eta^1,\dots,\eta^d)$ with $\eta^i \in \BR^d$. Using Assumption \ref{A:vf} and arguing as above, we obtain virtually identical estimates for $b$ and its derivatives with $\sfK_b$ in place of $\sfK_h$. We will repeatedly use the fact that $\ren(t)$ is non-decreasing with $\ren(\sfT)=\sfN < \infty$.

From the first equation in \eqref{E:JLMN}, we see that 
\begin{equation*}
\|\hat\JJ^\eps(t)\| \le \frac{1}{2}\sum_{i=1}^{\ren(t)} \left(\prod_{j=i+1}^{\ren(t)} \|Dh(x^-(t_j))\|_\sfF\right)  \left\|\la D^2 h(\hat{\bm\eta}^{i,\varepsilon})(X^{\eps, -}_{t_i} - x^{-}(t_i)),(X^{\eps, -}_{t_i} - x^{-}(t_i)) \ra\right\|.
\end{equation*}
Hence, we have $\|\hat\JJ^\eps(t)\| \le \frac12 \sum_{i=1}^{\ren(t)} (\sfK_h d)^{\ren(t)} \sfK_h d^2 \|X^{\eps,-}_{t_i}-x^-(t_i)\|^2$. Noting that the sum here only involves impulse times $t_i$ satisfying $t_i \le t$, it follows that there exists $C_1>0$ such that
\begin{equation}\label{E:J-estimate}
\|\hat\JJ^\eps(t)\| \le C_1 \sup_{0 \le s \le t} \|X^\eps_s-x(s)\|^2 \quad \text{for all $t \in [0,\sfT]$.}
\end{equation}
Next, we observe that
\begin{multline*}
\|\hat\LL^\eps(t)\| \le \frac{1}{2} \sum_{i=1}^{\ren(t)} \left(\prod_{j=i}^{\ren(t)} \|Dh(x^-(t_{j}))\|_\sfF \right) \int_{t_{i-1}}^{t_i} \left\|\la D^2 b(\hat{\bm\xi}^{\varepsilon}_s) (X^{\varepsilon}_{s}-x(s)),(X^{\varepsilon}_{s}-x(s)) \ra\right\| \thinspace ds 
 \\ + \frac{1}{2}\int_{t_{\ren(t)}}^{t} \left\|\la D^2 b(\hat{\bm\xi}^{\varepsilon}_s)(X^{\varepsilon}_{s}-x(s)), (X^{\varepsilon}_{s}-x(s)) \ra \right\| \thinspace ds
\end{multline*}
Estimating terms as we did for the case of $\hat\JJ^\eps(t)$, we get\\
 $\|\hat\LL^\eps(t)\| \le \frac12 \sum_{i=1}^{\ren(t)} \left(\sfK_h d\right)^{\ren(t)} \int_{t_{i-1}}^{t_i} \sfK_b d^2 \|X^\eps_s-x(s)\|^2 \thinspace ds + \frac12 \int_{t_{\ren(t)}}^t \sfK_b d^2 \|X^\eps_s-x(s)\|^2 \thinspace ds$. Hence, there exists a constant $C_2>0$ such that
\begin{equation}\label{E:L-estimate}
\|\hat\LL^\eps(t)\| \le C_2 \int_0^t \sup_{0 \le u \le s} \|X^\eps_u-x(u)\|^2 \thinspace ds \quad \text{for all $t \in [0,\sfT]$.}
\end{equation}
From the third equation in \eqref{E:JLMN}, we see that
\begin{multline*}
\|\hat\MM^\eps (t)\| \le   \sum_{i=1}^{\ren(t)} \left(\prod_{j=i}^{\ren(t)} \|Dh(x^-(t_{j}))\|_\sfF \right) \int_{t_{i-1}}^{t_i} \|Db(x(s))\|_\sfF \left\|X^\eps_{s}-x(s)-\eps Z_s\right\| \thinspace ds \\
+\int_{t_{\ren(t)}}^t \|Db(x(s))\|_\sfF \|X^\eps_{s}-x(s)-\eps Z_s\| \thinspace ds,
\end{multline*}
which yields $\|\hat\MM^\eps(t)\| \le \sum_{i=1}^{\ren(t)} \left(\sfK_h d\right)^{\ren(t)} \int_{t_{i-1}}^{t_i} \sfK_b d \|X^\eps_s - x(s) - \eps Z_s \| \thinspace ds + \int_{t_{\ren(t)}}^t \sfK_b d \|X^\eps_s - x(s) - \eps Z_s \| \thinspace ds$. Hence, there exists a constant $C_3>0$ such that 
\begin{equation}\label{E:M-estimate}
\|\hat\MM^\eps(t)\| \le C_3 \int_0^t \sup_{0 \le u \le s} \|X^\eps_u - x(u) - \eps Z_u\| \thinspace ds \quad \text{for all $t \in [0,\sfT]$.}
\end{equation}
Next, we compute that
\begin{multline*}
\|\hat\NN^\eps (t)\| \le  \eps \sum_{i=1}^{\ren(t)} \left(\prod_{j=i}^{\ren(t)} \|Dh(x^-(t_{j}))\|_\sfF\right) \left\|\int_{t_{i-1}}^{t_i} \left(\sigma(X^\eps_s) -\sigma(x(s))\right) \thinspace dW_s\right\| \\+\eps \left\|\int_{t_{\ren(t)}}^{t} \left(\sigma(X^\eps_s) -\sigma(x(s))\right) \thinspace dW_s\right\|
\end{multline*}
Using the fact that
\begin{equation*}
\left\| \int_s^t \left(\sigma(X^\eps_u)-\sigma(x(u)) \right) dW_u\right\| \le 2 \sup_{0 \le v \le t} \left\| \int_0^v \left(\sigma(X^\eps_u)-\sigma(x(u)) \right) dW_u\right\| \quad \text{for all $0 \le s < t \le \sfT$,}
\end{equation*}
we easily get 
$\|\hat\NN^\eps(t)\| \le 2 \eps \left(\sum_{i=1}^{\ren(t)} \left(\sfK_h d\right)^{\ren(t)} + 1 \right) \sup_{0 \le v \le t} \left\|\int_0^v \left(\sigma(X^\eps_u)-\sigma(x(u))\right) \thinspace dW_u\right\|$.
Hence there exists $C_4>0$ such that 
\begin{equation}\label{E:N-estimate}
\|\hat\NN^\eps(t)\| \le \eps C_4 \sup_{0 \le v \le t} \left\|\int_0^v \left(\sigma(X^\eps_u)-\sigma(x(u))\right) \thinspace dW_u\right\| \quad \text{for all $t \in [0,\sfT]$.}
\end{equation}
Putting together equations \eqref{E:J-estimate}--\eqref{E:N-estimate}, we easily get the stated claim.
\end{proof}

\begin{lemma}\label{L:diffn-diff-est}
For $\eps \in (0,1)$, $t \in [0,\sfT]$, we have
\begin{equation*}
\BE\left[ \sup_{0 \le v \le t} \left\| \int_0^v \left(\sigma(X^\eps_u)-\sigma(x(u))\right) \thinspace dW_u \right\|\right] \le 2 \sfK_\sigma r \left\{\BE\left[\int_0^t \sup_{0 \le u \le s}\|X^\eps_u-x(u)\|^2 \thinspace ds \right]\right\}^{1/2}.
\end{equation*}
\end{lemma}

\begin{proof}[Proof the Lemma]
For $1 \le i \le d$, $1 \le j \le r$, set
\begin{multline*}
\sfM^\eps_{ij}(t) \triangleq \int_0^t \left(\sigma_{ij}(X^\eps_u)-\sigma_{ij}(x(u))\right) \thinspace dW^j_u \quad \text{and note that} \\ \left\|\int_0^t \left(\sigma(X^\eps_u)-\sigma(x(u))\right) \thinspace dW_u\right\|^2 = \sum_{i=1}^d \left(\sum_{j=1}^r \sfM^\eps_{ij}(t)\right)^2 \le \sum_{i=1}^d \left(\sum_{j=1}^r |\sfM^\eps_{ij}(t)| \right)^2 \le r^2 \sum_{i=1}^d \sum_{j=1}^r |\sfM^\eps_{ij}(t)|^2
\end{multline*}
Using Doob's maximal inequality, followed by the Ito isometry, we get 
\begin{equation*}
\BE\left[\sup_{0 \le v \le t}|M^\eps_{ij}(v)|^2\right] \le 4\BE\left[|M^\eps_{ij}(t)|^2\right] = 4\BE\left[\int_0^t |\sigma_{ij}(X^\eps_s)-\sigma_{ij}(x(s))|^2 \thinspace ds\right]
\end{equation*}
This yields
\begin{multline*}
\BE\left[ \sup_{0 \le v \le t} \left\| \int_0^v \left(\sigma(X^\eps_u)-\sigma(x(u))\right) \thinspace dW_u \right\|^2\right] \le 4r^2 \sum_{i=1}^d \sum_{j=1}^r \BE\left[\int_0^t \left(\sigma_{ij}(X^\eps_s)-\sigma_{ij}(x(s))\right)^2 \thinspace ds\right] \\ = 4r^2 \BE\left[ \int_0^t \|\sigma(X^\eps_s)-\sigma(x(s))\|_\sfF^2 \thinspace ds\right] \le 4 \sfK_\sigma^2 r^2 \BE\left[\int_0^t \sup_{0 \le u \le s}\|X^\eps_u-x(u)\|^2 \thinspace ds \right]
\end{multline*}
where the latter follows using the Lipschitz continuity of $\sigma$ as in Assumption \ref{A:vf}. Recalling that $\BE[|Z|] \le (\BE[|Z|^2])^{1/2}$ for a square integrable random variable $Z$, we now get the stated result.
\end{proof}

We now provide the proof of Theorem \ref{T:CLT}.

\begin{proof}[Proof of Theorem \ref{T:CLT}]
Taking expectations in equation \eqref{E:CLT-comparison}, we get
\begin{multline*}
\BE\left[\sup_{0 \le s \le t}\|X^\eps_{s}-x(s)- \eps Z_s\|\right]  \le C_{\ref{L:CLT-comparison}}\left\{ \BE\left[\sup_{0 \le s \le t}\|X^{\varepsilon}_s-x(s)\|^2\right] 
+ \int_{0}^{t}  \BE\left[\sup_{0\le u\le s}\|X^{\varepsilon}_{u}-x(u)\|^2 \right]\thinspace ds\right. \\  \qquad \left. + \int_0^t \BE\left[\sup_{0 \le u \le s}\|X^\eps_{u}-x(u)-\eps Z_u\|\right] \thinspace ds
+ \eps 2 \sfK_\sigma r \left(\int_0^t \BE\left[\sup_{0 \le u \le s}\|X^\eps_u-x(u)\|^2\right] \thinspace ds \right)^{1/2}\right\}. 
\end{multline*}
Since $\BE\left[\sup_{0 \le s \le t}\|X^\eps_s-x(s)\|^2\right] \le C_{\ref{T:LLN}} \eps^2$, we see that there exists $C>0$ such that
\begin{equation*}
\BE\left[\sup_{0 \le s \le t}\|X^\eps_{s}-x(s)- \eps Z_s\|\right]  \le C\left\{\eps^2 + \int_0^t \BE\left[\sup_{0 \le u \le s}\|X^\eps_{u}-x(u)-\eps Z_u\|\right] \thinspace ds\right\}
\end{equation*}
for all $t \in [0,\sfT]$, $\eps \in (0,1)$. Gronwall's inequality now yields the stated claim \eqref{E:CLT-rate}.
\end{proof}


\section{Numerical example and simulation}\label{S:Example}

In this section, we numerically illustrate our primary results, Theorems \ref{T:LLN} and \ref{T:CLT}, using two different examples: the periodically kicked nonlinear pendulum with state-dependent kick sizes, and the FitzHugh-Nagumo system subject to periodic kicks. 
Below, we will frequently indicate time-dependence of the stochastic process $Y_t$ using the notation $Y(t)$, primarily to reserve subscripts for denoting various components of the vector $Y(t)$, e.g., $Y_i(t)$ for the $i$-th component.

\subsection{Periodically kicked pendulum}
The dynamics of the undamped pendulum are governed by 
\begin{equation}\label{E:Pendulum}
    \begin{aligned}
\dot{x}_1 &= x_2, \\
\dot{x}_2 &= -\alpha\sin x_1
\end{aligned}
\end{equation}
where $x_1$ and $x_2$ represent the angular position and velocity of
the pendulum, respectively, and $\alpha>0$ is a constant.
In line with the problem hypotheses, we use the resetting law
\begin{equation}\label{E:resetting-pendulum}
h(x_1,x_2)
\triangleq (x_1, x_2+0.1 \sin{x_1})^{\top};
\end{equation}
the choice of $h$ in \eqref{E:resetting-pendulum} is motivated by the map used in \cite{LinYoung_JFPTA} to study periodically kicked linear shear flow with a hyperbolic limit cycle, adapted to be consistent with Assumption \ref{A:vf} (Lipschitz continuity and linear growth). 
If the impulses are assumed to arrive in a time periodic manner at times $\{t_n\}_{n=1}^\infty$, then with $x = (x_1,x_2) \in \mathbb{R}^2$, the system described in \eqref{E:Pendulum}, \eqref{E:resetting-pendulum} matches \eqref{E:sie-det-per} with $
b(x) = (x_2, -\alpha\sin x_1)^{\top}$, 
$h(x) = (x_1, x_2+0.1 \sin{x_1})^{\top}$.
For simplicity, we set the diffusion matrix $\sigma$ to be the $2 \times 2$ identity matrix. 

\medskip
Having set the stage, the trajectories $x(t)$, $X_t^{\varepsilon}$, and $Z_t$, which are governed by \eqref{E:det-state-int-eq}, \eqref{E:stoch-state-int-eq}, and \eqref{E:fluct-proc}, respectively, are computed numerically in \textsc{matlab}.
Using the the Euler--Maruyama method \cite{higham2001algorithmic, KloedenPlaten}, we generate 1000 sample paths of $X_t^{\varepsilon}$
and $Z_t$ with parameter values $\Delta t = 2^{-12}$, $\sfT = 8$, $\alpha= 1$,  and initial conditions $X_{1}^{\varepsilon}(0) = x_1(0) = 0.5$, $X_{2}^{\varepsilon}(0) = x_2(0) = 0.5$, and $Z_1(0) = Z_2(0) = 0$. As the obtained result Theorem \ref{T:CLT} on approximating $X_t^{\varepsilon}$ by $A^\eps_t \triangleq x(t) + \varepsilon Z(t)$ is valid in a path-wise sense, the same Brownian increments have been used for $X_t^{\varepsilon}$ and $Z_t$.

First, Figures 1 and 2 show behavior of individual sample paths for purposes of visualization and preliminary comparison. In Figure 1, a sample trajectory of  $x_i(t)$, $X^\eps_i(t)$, $A^\eps_i(t) \triangleq x_i(t) + \eps Z_i(t)$, $i \in \{1,2\}$, with $\varepsilon = 2^{-4}$, is generated and the error computed as a function of time, showing good agreement between the true stochastic trajectory $X^\eps_i(t)$ and its first order expansion $A^\eps_i(t)$. 
It is worth noting that the sample path of the approximating process $A^\eps_t$ traces the path of $X^\eps_t$ more closely than the deterministic path $x(t)$; this is consistent with the error sizes of $\eps$ and $\eps^2$ in Theorems \ref{T:LLN} and \ref{T:CLT}, respectively. 
Next, Figure 2 shows the probabilistic aspect of our result through a comparison of the pre-limit fluctuation process $Y_i^{\eps}(t) \triangleq \frac{X_i^{\varepsilon}(t)- x_i(t)}{\eps}$ against the limiting fluctuation process $Z_i(t)$ for $\varepsilon = 2^{-5}$. 
In both cases, we note that on account of our specific choice of $h$ in \eqref{E:resetting-pendulum}, the trajectories of the first component are continuous at the impulse times (represented by solid dots) with only the second component displaying jumps.

Finally, to explore the dependence of the errors in Theorems \ref{T:LLN} and \ref{T:CLT} on $\eps$, 
we consider varying values of $\varepsilon = 2^{-i}$, $1 \le i \le 10$, and compute the empirical mean over 1000 sample paths of the errors $ \sup_{0\le t \le \sfT} |X_j^{\varepsilon}(t) - x_j(t)|$ and $ \sup_{0\le t \le \sfT} |X_j^{\varepsilon}(t) - A^\eps_j(t)|$, $j=1,2$. The resulting quantities, denoted by $(e_1,e_2)$ in both cases, are plotted in Figure 3 on a $\log_2$--$\log_2$ scale. The slopes are found to be approximately $1$ and $2$, consistent with the theoretical results.

%

\begin{figure}[ht]\label{F:Figure1}
    \centering
    \begin{subfigure}{0.48\textwidth}
        \centering
        \includegraphics[width=1.15\linewidth]{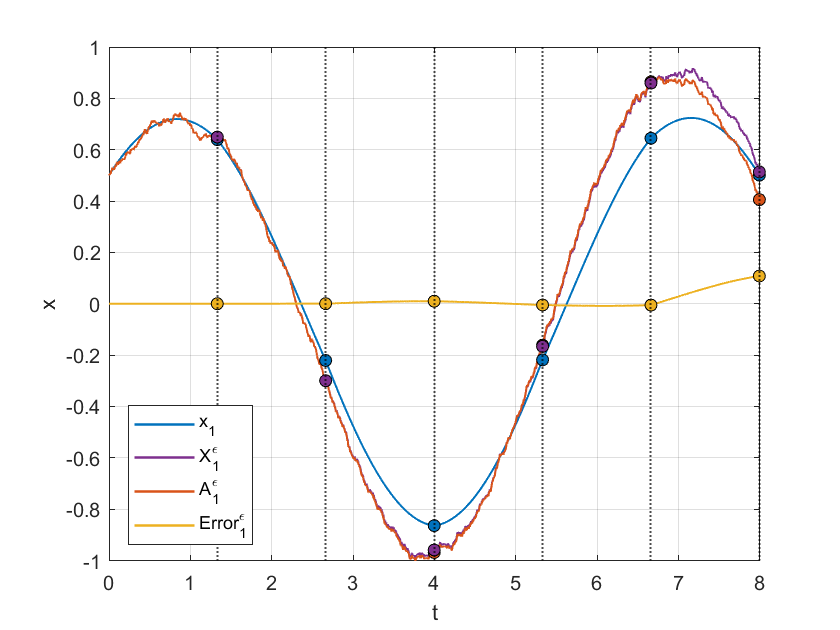}
        \caption{First coordinate}
    \end{subfigure}
    \hfill
    \begin{subfigure}{0.48\textwidth}
        \centering
        \includegraphics[width=1.15\linewidth]{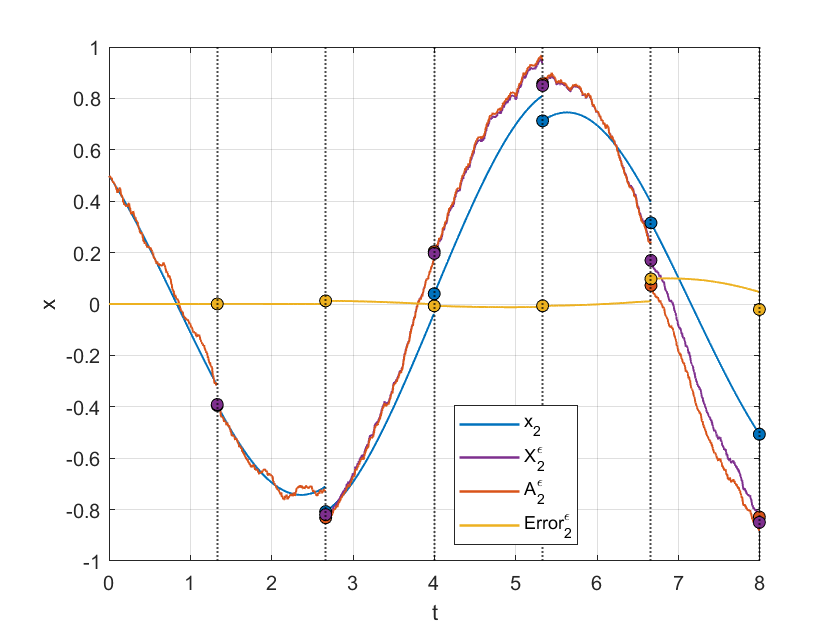}
        \caption{Second coordinate}
    \end{subfigure}
    \caption{Sample paths of the components $X_i\triangleq X^\eps_i(t)$, $x_i(t)$,  $A^\eps_i\triangleq x_i+\eps Z_i(t)$ and $Error^\eps_i \triangleq X^\eps_i - A^\eps_i$ for $i=1, 2$ and $\eps=2^{-4}$.}
\end{figure}

\begin{figure}[ht]
    \centering
    \begin{subfigure}{0.48\textwidth}
        \centering
        \includegraphics[width=1.15\linewidth]{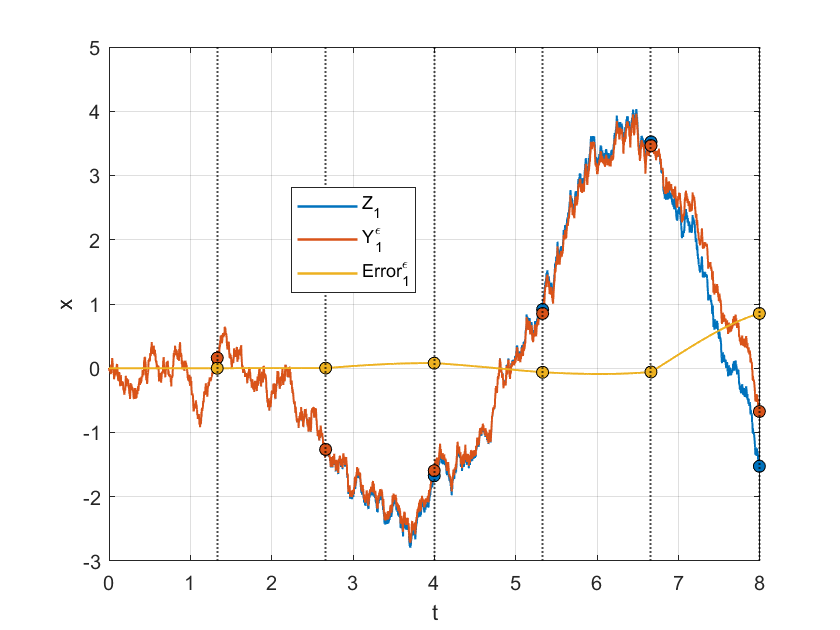}
        \caption{First Coordinate}
    \end{subfigure}
    \hfill
    \begin{subfigure}{0.48\textwidth}
        \centering
        \includegraphics[width=1.15\linewidth]{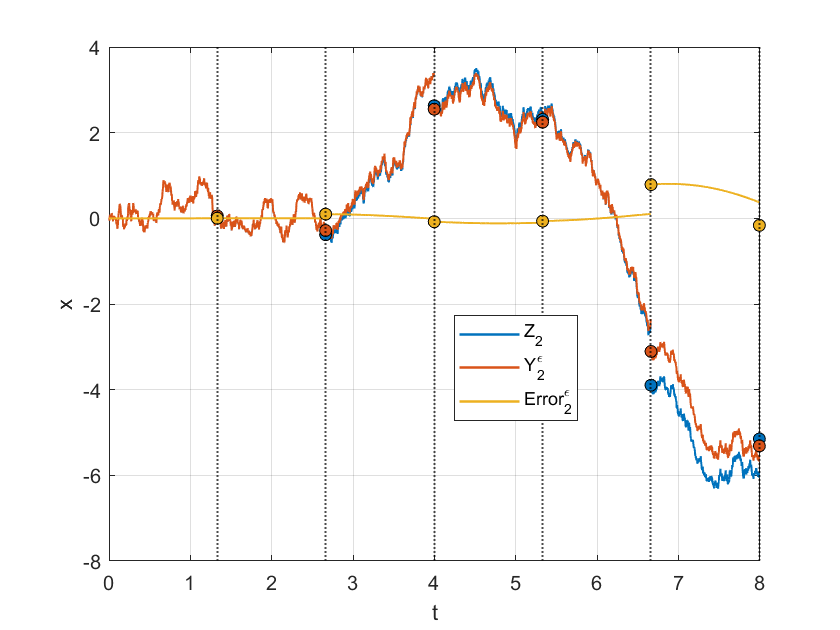}
        \caption{Second Coordinate}
    \end{subfigure}
    \caption{Sample paths of the components $Y_i\triangleq \frac{X^\eps_i(t)- x_i(t)}{\eps}$, $Z_i(t)$ and $Error^\eps_i \triangleq Y^\eps_i - Z_i$ for $i=1, 2$ and $\eps=2^{-5}$ .}
\end{figure}
\begin{figure}[ht]
    \centering
    \begin{subfigure}{0.48\textwidth}
        \centering
        \includegraphics[width=1.15\linewidth]{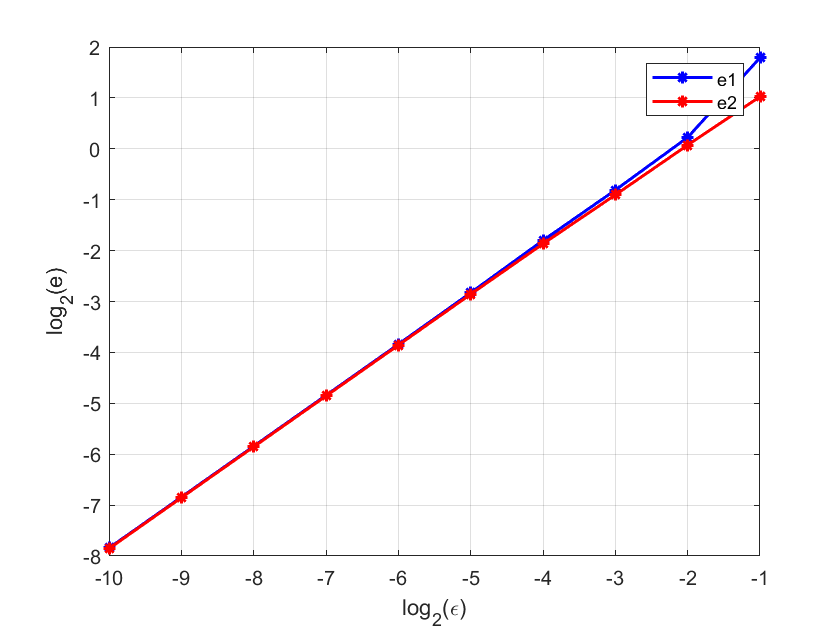}
        \caption{LLN Error in both coordinates}
    \end{subfigure}
    \hfill
    \begin{subfigure}{0.48\textwidth}
        \centering
        \includegraphics[width=1.15\linewidth]{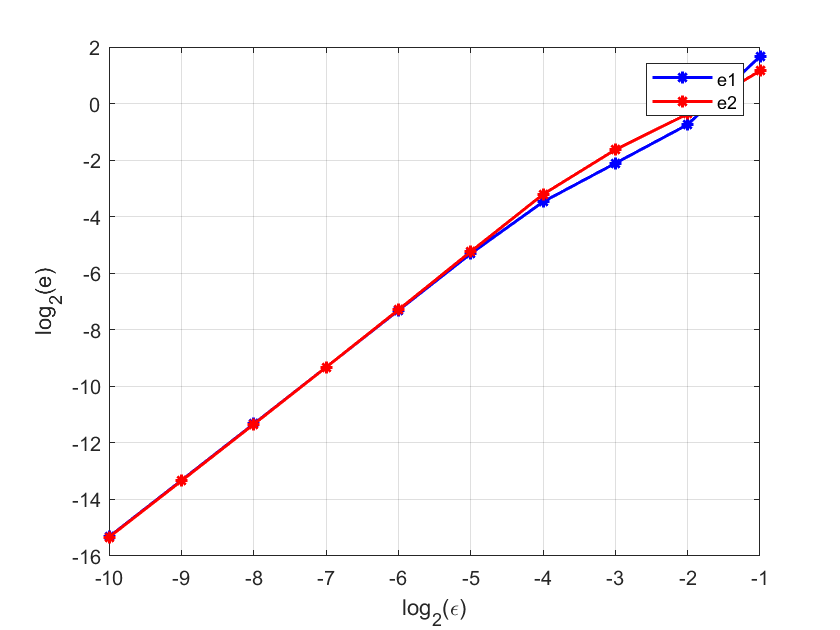}
        \caption{CLT Error in both coordinates}
    \end{subfigure}
    \caption{In the figure (A) the components $(e_1, e_2)$ which are the mean of $\sup_{0\le t\le \sfT}|X^\eps_i(t) - x_i(t)|$ and in the figure (B) the components $(e_1, e_2)$ which are the mean of $\sup_{0\le t\le \sfT}|X^\eps_i(t) - A^\eps_i(t)|$,  over $1000$ sample paths are plotted on a $\log_2$--$\log_2$ scale. The values of $e_i, i=1, 2$ decrease with increasing $i$, where the values of $\eps=2^{-i}, 1\le i\le 10$.}
   
\end{figure}

\subsection{FitzHugh-Nagumo System}
We now illustrate our results, Theorems \ref{T:LLN} and \ref{T:CLT}, in the context of a periodically kicked planar {\sc ode} motivated by the FitzHugh-Nagumo system in mathematical neuroscience \cite{ErmentroutTerman}, \cite{Izhikevich}. 
To get started, we note that while the Hodgkin-Huxley system is one of the best established models to describe the generation of action potentials in neurons, it involves a four-dimensional system and can sometimes be challenging to work with.
As a result, planar simplifications of the form 
\begin{equation}\label{E:planar-neuro}
\begin{aligned}
\dot{x}_1 &= f(x_1,x_2), \\
\dot{x}_2 &= g(x_1,x_2),
\end{aligned}
\end{equation}
which retain the essential qualitative dynamics of the full Hodgkin-Huxley equations, have been of great interest.
One such simplification is given by the FitzHugh-Nagumo system \cite{kaumann2006uniqueness}, \cite{AmbrosioMintchev} (see below) where the variables $x_1$ and $x_2$ represent, respectively, the membrane potential and recovery variable.

Taking our cue from \cite{cebrian2024six}, we consider the {\sc ode} \eqref{E:planar-neuro} where the vector fields $f,g$ satisfy Assumption \ref{A:vf} on the plane $\BR^2$, and further
\begin{equation}\label{E:FHN-vf}
\left(f(x_1,x_2),g(x_1,x_2)\right) = \left(\frac{1}{C}( x_1- \frac{x^3_1}{3} +x_2),  \frac{D}{C}(- x_1 - c - bx_2)\right)^{\top} \qquad \text{for $x_1^2 + x_2^2 < R^2$ with $R \gg 1$,}
\end{equation}
where $C, D, c$ and $b$ are system parameters.
We note that the {\sc ode} \eqref{E:planar-neuro} with $f,g$ given by \eqref{E:FHN-vf} behaves exactly like the FitzHugh-Nagumo system in a large neighborhood of the origin.\footnote{We have modified the FitzHugh-Nagumo vector field only outside a large ball of radius $R \gg 1$; this is to ensure global compliance with our problem assumptions, while still retaining the relevant local features near the origin.}
Following \cite{AmbrosioMintchev}, we consider the case where the recovery variable $x_2$ is subjected to periodic kicks of constant magnitude $-A$ (here, $A \in \BR$) leading to the resetting map 
\begin{equation}\label{E:resetting-FHNnew}
h(x_1,x_2)
\triangleq (x_1, x_2-A)^{\top}.
\end{equation}
Assuming the impulses arrive in a time periodic manner at times $\{t_n\}_{n=1}^\infty$, our deterministic impulsive system of interest $x(t)=(x_1(t),x_2(t))$ is obtained by combining \eqref{E:planar-neuro}, \eqref{E:FHN-vf} and \eqref{E:resetting-FHNnew} according to the recipe in \eqref{E:det-state-int-eq}. 
The stochastic process $X^\eps(t) = (X^\eps_1(t),X^\eps_2(t))$ is now obtained in accordance with \eqref{E:stoch-state-int-eq} by perturbing $x(t)$ above by a two-dimensional Brownian motion of size $0<\eps \ll 1$; for simplicity, we take the diffusion matrix $\sigma$ to be the identity.

The trajectories $x(t)$, together with 1000 sample paths of $X^{\varepsilon}(t)$ and $Z(t)$ for the system above, where $Z(t)$ is obtained according to \eqref{E:fluct-proc}, are numerically computed in \textsc{matlab} using the Euler--Maruyama method \cite{higham2001algorithmic, KloedenPlaten} with parameter values $\Delta t = 2^{-12}$, $\sfT = 8$, $C=0.02, D=0.2, c=-1.2, b=0.05$,  and $A=0.2$, and initial conditions $X_{1}^{\varepsilon}(0) = x_1(0) = 0.5$, $X_{2}^{\varepsilon}(0) = x_2(0) = 0.5$, and $Z_1(0) = Z_2(0) = 0$.
Since the results in Theorem \ref{T:CLT} center on approximating $X^{\varepsilon}(t)$ by $A^\eps(t) \triangleq x(t) + \varepsilon Z(t)$ in a path-wise sense, the same Brownian increments have been used for $X^{\varepsilon}(t)$ and $Z(t)$.

Figures  \ref{F:Figure4A} and \ref{F:Figure4B} describe the behavior of individual (typical) sample paths and facilitate visualization together with a preliminary comparison: 
a sample trajectory of  $x_i(t)$, $X^\eps_i(t)$, $A^\eps_i(t) \triangleq x_i(t) + \eps Z_i(t)$, $i \in \{1,2\}$, with $\varepsilon = 2^{-4}$, is generated and the error computed as a function of time, exhibiting close agreement between the true stochastic trajectory $X^\eps_i(t)$ and its first order expansion $A^\eps_i(t)$. 
It is worth noting that the approximating process $A^\eps(t)$ tracks the sample path of $X^\eps(t)$ more closely than does the deterministic path $x(t)$; this is consistent with the error sizes of $\eps$ and $\eps^2$ in Theorems \ref{T:LLN} and \ref{T:CLT}, respectively. 
We also note that by virtue of the particular choice of $h$ in \eqref{E:resetting-FHNnew}, the first-component trajectories are continuous across the impulse times (denoted by solid dots), whereas the impulses induce jumps exclusively in the second component.

In a similar spirit to the preceding pendulum example, we finally investigate the dependence of the ensemble errors in Theorems \ref{T:LLN} and \ref{T:CLT} on $\eps$. With different values of $\varepsilon = 2^{-i}$, $3 \le i \le 10$, we compute the empirical mean over 1000 sample paths of the errors $ \sup_{0\le t \le \sfT} |X_i^{\varepsilon}(t) - x_i(t)|$ and $ \sup_{0\le t \le \sfT} |X_i^{\varepsilon}(t) - A^\eps_i(t)|$, $i=1,2$. The resulting quantities, denoted by $(e_1,e_2)$ in both cases, are plotted in Figure \ref{F:Figure5} on a $\log_2$--$\log_2$ scale. Once again, the slopes are observed to be approximately $1$ and $2$, as predicted by the theoretical results.

\begin{figure}[ht]\label{F:Figure4}
    \centering
    \begin{subfigure}{0.48\textwidth}
        \centering
        \includegraphics[width=1.15\linewidth]{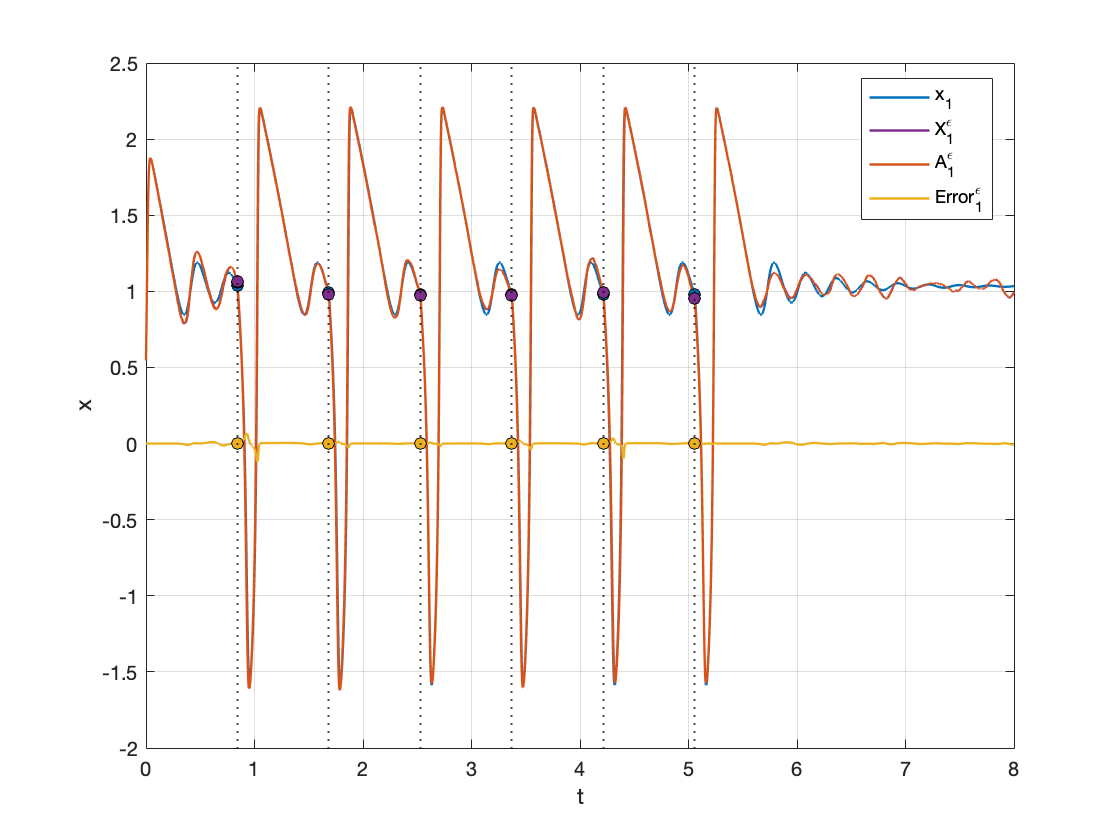}
        \caption{First coordinate}
        \label{F:Figure4A}
    \end{subfigure}
    \hfill
    \begin{subfigure}{0.48\textwidth}
        \centering
        \includegraphics[width=1.15\linewidth]{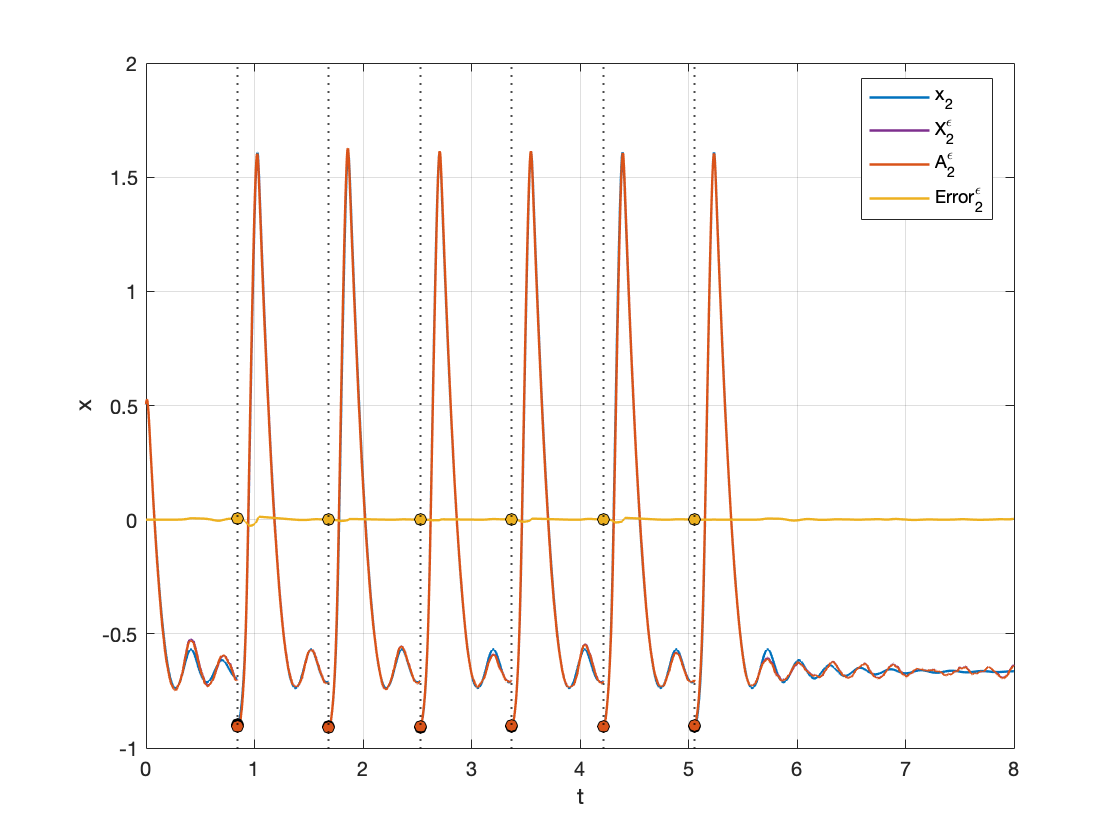}
        \caption{Second coordinate}
        \label{F:Figure4B}
    \end{subfigure}
    \caption{Sample paths of the components $X_i\triangleq X^\eps_i(t)$, $x_i(t)$,  $A^\eps_i(t)\triangleq x_i(t)+\eps Z_i(t)$ and $Error^\eps_i \triangleq X^\eps_i - A^\eps_i$ for $i=1, 2$ and $\eps=2^{-5}$.}
    \label{F:Figure4}
\end{figure}

\begin{figure}[ht]
    \centering
    \begin{subfigure}{0.48\textwidth}
        \centering
        \includegraphics[width=1.15\linewidth]{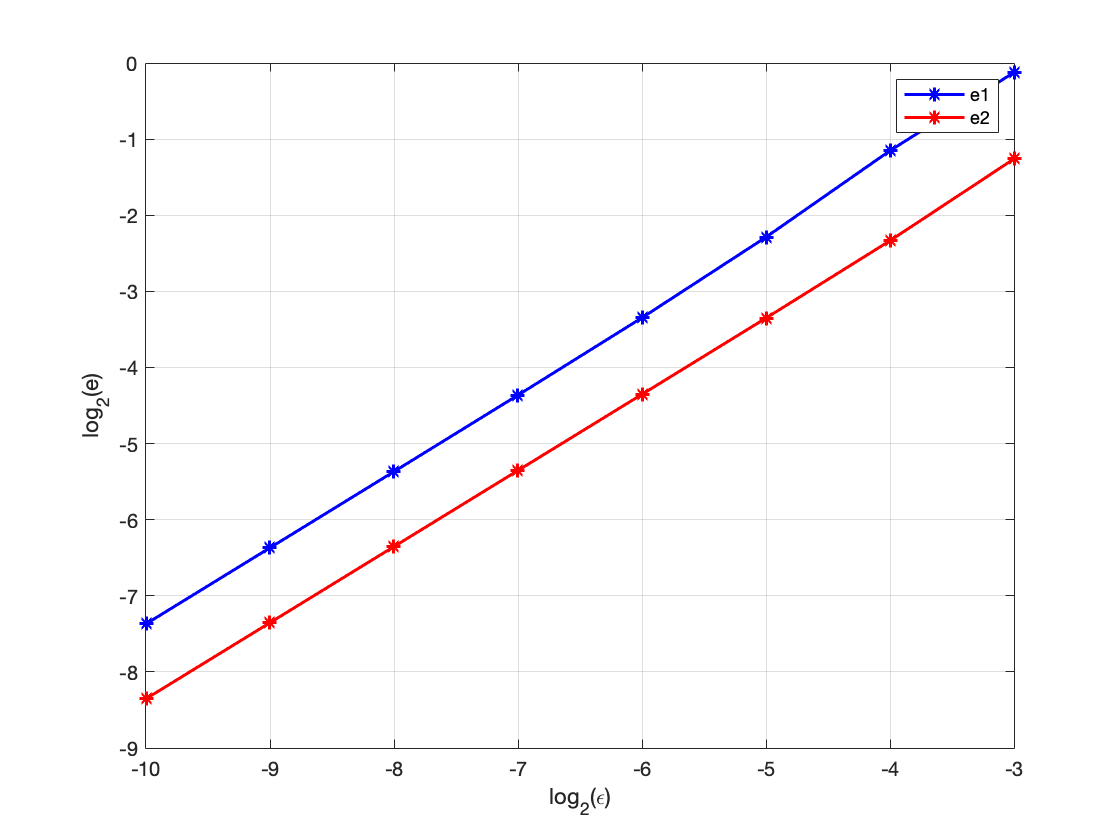}
        \caption{LLN Error in both coordinates}
    \end{subfigure}
    \hfill
    \begin{subfigure}{0.48\textwidth}
        \centering
        \includegraphics[width=1.15\linewidth]{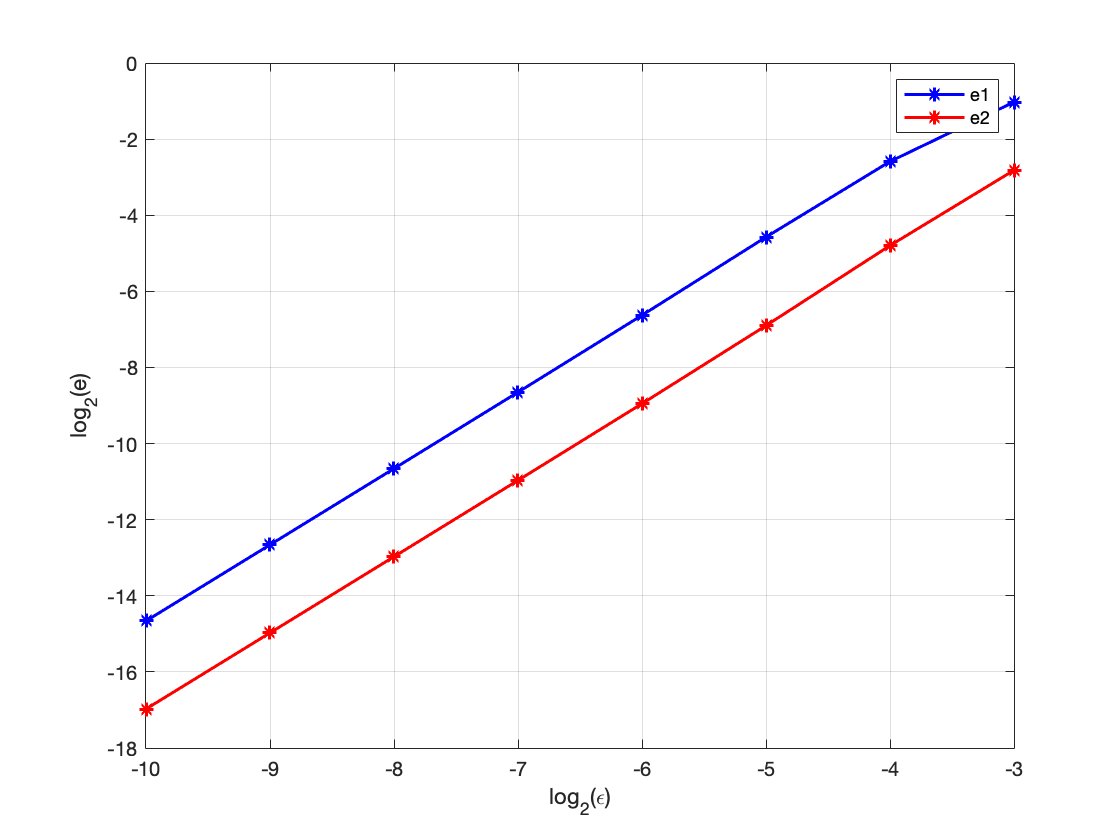}
        \caption{CLT Error in both coordinates}
    \end{subfigure}
    \caption{In the figure (A), the components $(e_1, e_2)$ which are the mean of $\sup_{0\le t\le \sfT}|X^\eps_i(t) - x_i(t)|$, $i \in \{1,2\}$, and in the figure (B), the components $(e_1, e_2)$ which are the mean of $\sup_{0\le t\le \sfT}|X^\eps_i(t) - A^\eps_i(t)|$,  over $1000$ sample paths are plotted on a $\log_2$--$\log_2$ scale. The values of $e_i, i=1, 2$ decrease with increasing $i$, where $\eps=2^{-i}, 3\le i\le 10$.}
   \label{F:Figure5}
\end{figure}

\bibliographystyle{alpha}
\bibliography{ImpactsLiterature}

\end{document}